\documentclass[11pt,a4paper,openany]{amsart}

\usepackage{amsmath,amsfonts,ams}
\usepackage{latexsym}
\usepackage{amssymb}
\usepackage{color}
\usepackage{epsfig}
\usepackage{subfigure}
\usepackage{mathrsfs}
\usepackage{amscd}
\usepackage{amsthm}
\usepackage{color}
\usepackage{cite}

\numberwithin{equation}{section}
\newtheorem{proposition}{Proposition}[section]

\newtheorem{lemma}{Lemma}[section]
\newtheorem{theorem}{Theorem}[section]
\newtheorem{corollary}{Corollary}[section]
\newtheorem{remark}{Remark}[section]

\begin{document}
\title[Linear restriction estimates on metric cones]
{Linear  restriction estimates for Schr\"odinger equation on metric
cones}

\author{Junyong Zhang}
\address{Department of Mathematics, Beijing Institute of Technology, Beijing
100081, China and Department of Mathematics, Australian National
University, Canberra, ACT 0200, Australia}
\email{zhang\_junyong@bit.edu.cn}

\maketitle

\begin{abstract}
In this paper, we study some modified linear restriction estimates
of the dynamics generated by Schr\"odinger operator on metric cone
$M$, where the metric cone $M$ is of the form
$M=(0,\infty)_r\times\Sigma$ with the cross section $\Sigma$ being a
compact $(n-1)$-dimensional Riemannian manifold $(\Sigma,h)$ and the
equipped metric is $g=\mathrm{d}r^2+r^2h$. Assuming the initial data
possesses additional regularity in angular variable
$\theta\in\Sigma$, we show some linear restriction estimates for the
solutions. As applications, we obtain global-in-time Strichartz
estimates for radial initial data and show small initial data
scattering theory for the mass-critical nonlinear Schr\"odinger
equation on two-dimensional metric cones.
\end{abstract}

\begin{center}
 \begin{minipage}{120mm}
   { \small {\bf Key Words: Linear restriction estimate, Metric cone, Strichartz estimates}
      {}
   }\\
    { \small {\bf AMS Classification:}
      { 42B37, 35Q40, 47J35.}
      }
 \end{minipage}
 \end{center}
\section{Introduction and Statement of Main Result}

We study some restriction estimates for the solution of
Schr\"odinger equations on the setting of metric cone. The metric
cone $M$ is of the form $M=(0,\infty)_r\times\Sigma$, where
$(\Sigma,h)$ is a compact $(n-1)$-dimensional Riemannian manifold
and the metric of $M$ is $g=\mathrm{d}r^2+r^2h$. More precisely, we
consider solutions $u: \R\times M\rightarrow \C$ to the initial
problem (IVP) for the Schr\"odinger equation on $M$,
\begin{equation}\label{1.1}
i\partial_t u(t,z)+H u(t,z)=0,\quad u(t,z)|_{t=0}=u_0(z),\quad
(t,z)\in\R\times M.
\end{equation}
Here, we use the operator $H=-\Delta_g+q(\theta)/r^2$ where
$\Delta_g$ denotes the Friedrichs extension of Laplace-Beltrami from
the domain $C_c^\infty(M^\circ)$, compactly supported smooth
functions on the interior of the metric cone, and we write
$q(\theta)$ for a smooth function on $\Sigma$ such that
$-\Delta_h+q(\theta)$ is positive on $L^2(\Sigma)$. The Euclidean
space $\R^n$ is the simplest example of a metric cone; its cross
section is $(\mathbb{S}^{n-1},\mathrm{d}\theta^2)$. We note that the
general metric cones have a dilation symmetry analogous to that of
Euclidean space but no other symmetries in general.\vspace{0.2cm}

There is a large amount of literature focused on the restriction
theory on the Euclidean space,  we refer the readers to
\cite{Barcelo, Str, Tao, Tao2, Tao3, TVV, Wolff}. Shao
\cite{Shao,Shao1} proved the cone and parabolic restriction
conjectures hold true for the spatial rotation invariant functions
which are supported on the cone or parabola. Motivated by
\cite{Shao,Shao1}, Miao, Zheng and the author\cite{MZZ,MZZ1}
utilized the spherical harmonics expansion and analyzed the
asymptotic behavior of the Bessel function to generalize Shao's
results by establishing restriction estimates with some angular
regularity loss. Based on \cite{MZZ}, Miao, Zheng and the author
\cite{MZZ2} proved a scale of Strichartz estimates (extending the
admissible restriction) for wave equation with an inverse square
potential when the initial data had additional angular regularity.
\vspace{0.2cm}

We are interested in the restriction estimate for the solution of
Schr\"odinger equations on the metric cone. Cones were studied from
the problem of wave diffraction from a cone point; see
\cite{Som,Frie,Frie1}. The Laplacian defined on cones has been
studied by Cheeger and Taylor \cite{CT,CT1}. Other aspects on the
metric cone also have been studied; for example the heat kernel and
Riesz transform kernel were studied in \cite{HL,Mo}. There has been
a lot of interest in the study of the Schr\"odinger propagator on
the smooth asymptotically conic Riemannian manifolds. We refer the
reader to Hassell, Tao and Wunsch \cite{HTW, HTW1} and Mizutani
\cite{Miz}. In particular, Guillarmou, Hassell and Sikora \cite{GHS}
showed a estimate of the spectral measure to obtain a Stein-Tomas
restriction theorem in this asymptotically conic setting. The
restriction problem is much more than the Stein-Tomas type
restriction estimates. We recall that a asymptotically conic
manifold $X$, outside some compact set, is isometric to a conical
space $M=\R_+\times \Sigma$, where $\Sigma$ is a compact
$(n-1)$-dimensional manifold with or without boundary. By analogy
with Euclidean space, we call $r\in \R_+$ the radial variable and
$\theta\in\Sigma$ the angular variable. Then $(r,\theta)$ are polar
coordinates on $M$, and we can write the metric as
$g=\mathrm{d}r^2+r^2 h$ with the Riemannian metric $h$ on $\Sigma$.
We refer the reader to \cite{Mel,HW} for more details on the
scattering manifolds.  Most arguments applying to metric cones can
be recognized as an ingredient of the analysis on asymptotically
conic manifolds. The problems on metric cones appear as model
problems when dealing with similar questions on asymptotically conic
manifolds.  We however will prove much more restriction estimates
than \cite{GHS} by assuming the initial data having additional
``angular" regularity. As applications,  we show a global-in-time
Strichartz estimate for the Schr\"odinger equation on the metric
cone for radial initial data. For two-dimensional metric cone, Ford
\cite{Ford} proved the full range of global-in-time Strichartz
estimates. We remark that the Strichartz estimates established in
\cite{HTW, HTW1, Miz} for scattering manifolds are local in time.
\vspace{0.2cm}

As pointed out in \cite{GHS}, the Laplacian on the scattering
manifolds gives rise to a family of Poisson operators $P(\lambda)$
defined for $\lambda>0$. The corresponding extension-restriction
problem is to consider the boundedness of $P(\lambda)$:
$L^p(\partial M)\rightarrow L^q(M)$. Its norm is in terms of the
frequency $\lambda$. The restriction conjecture on the ball and the
parabolic surface with dimension $n$ says that $1\leq
p<{2(n+1)}/{n}$ and ${(n+2)}/q\leq {n}/{p'}$ is a necessary and
sufficient condition. It is very hard to show the sufficient part
when $p$ is close to ${2(n+1)}/{n}$ and the problem still remains
open. \vspace{0.2cm}

In this paper, we follow the argument in \cite{MZZ1,MZZ2} to show
modified restriction estimates with some loss of angular regularity
for the solution of Schr\"odinger equation on conic manifold when
$p$ is close to ${2(n+1)}/{n}$. Since we do not know how to
construct an approximate ``global" parametrix for the propagator
$e^{itH}$, we have to write the propagator as a linear combination
of products of the Hankel transform of the radial part and
eigenfunctions of $-\Delta_h+q(\theta)$, the Laplace-Beltrami
operator on $\Sigma$. Though this expression may cause some loss of
angular regularity, it gives a global in time expression of the
solution. Compared with our previous work \cite{MZZ,MZZ2} for wave
equation, we need to exploit effectively the oscillation of the
multiplier $e^{ it\rho^2}$ which has much more oscillation than the
wave multiplier $e^{ it\rho}$ at high frequency. The Bessel function
$J_{\nu}(r)$ appears in the Hankel transform, and the decay property
of the Bessel function plays a key role in our argument. Since
$J_{\nu}(r)$ decays more slowly than $r^{-1/2}$ when $1\ll r\sim
\nu$, we overcome this difficulty by exploiting the oscillations
both in $e^{it\rho^2}$ and the Bessel function $J_{\nu}(r\rho)$ in
proving a localized estimate for $q=\infty$; see Proposition
\ref{linear estimates} below. However the strategy breaks down for
the other general $q$, for example $q=4$. We need develop the
advantage of the parabolic curvature. To do this, we use a bilinear
argument which is in spirit of Carleson-Sj\"olin argument or
equivalently the $TT^*$ method. In the process of using bilinear
argument, we have to divide into two cases $1\ll\nu\sim r\ll \nu^2$
and $\nu^2\ll r$. In the former, the low decay of Bessel function
leads to a loss of angular regularity. The latter will be treated by
using a complete asymptotic formula for the Bessel function in
\cite{Stein1,Watson}. The quantity $\nu^2$ is chosen to balance the
two things: the smallest loss of angular regularity and  the
absolutely convergent of the series of the coefficients in the
complete asymptotic formula. In the proof of the case $q=4$, we
additionally require a Whitney-type decomposition argument because
of the failure of Hardy-Littlewood-Sobolev inequality.\vspace{0.2cm}

To state our main result, we need some notation. Let
\begin{equation}\label{1.2} \chi_\infty=\Big\{\nu:
\nu=\sqrt{\lambda+(1/4)(n-2)^2},\quad \lambda~\text{is eigenvalue
of}-\widetilde{\Delta}_h:=-\Delta_h+q(\theta)\Big\},
\end{equation} and let
$d(\nu)$ be the multiplicity of $\lambda_\nu=\nu^2-\frac14(n-2)^2$
as eigenvalue of $-\widetilde{\Delta}_h$ and
$\{\varphi_{\nu,\ell}\}_{1\leq \ell\leq d(\nu)}$ the associated
eigenfunctions of $-\widetilde{\Delta}_h$. We then have the
decomposition of $f\in L^2(M)$
\begin{equation}\label{1.3}
f(z)=f(r,\theta)=\sum\limits_{\nu\in\chi_\infty}\sum\limits_{\ell=1}^{d(\nu)}a_{\nu,\ell}(r)\varphi_{\nu,\ell}(\theta).
\end{equation}
For more details, we refer to Section 2. We now define the
``distorted" Fourier transform of the Schwartz function $f$ by
\begin{equation}\label{1.4}
\mathcal{F}_{H}(f)(\rho,\omega)=\sum_{\nu\in\chi_\infty}\sum_{\ell=1}^{d(\nu)}\varphi_{\nu,\ell}(\omega)\int_0^\infty(r\rho)^{-\frac{n-2}2}
J_{\nu}(r\rho)a_{\nu,\ell}(r)r^{n-1}\mathrm{d}r,
\end{equation}
where $\omega\in\Sigma$ and $J_{\nu}(r)$ is the Bessel function of
order $\nu$. We remark that when $\Sigma=\mathbb{S}^{n-1}$,
$\varphi_{\nu,\ell}$ is the spherical harmonics function
$Y_{k,\ell}(\theta)\in L^2(\mathbb{S}^{n-1})$ of order $k$ and
$\nu=k+(n-2)/2$, then the ``distorted" Fourier transform defined
above, up to some constant, is same as the classical Fourier
transform by  \cite[Theorem 3.10]{SW}.\vspace{0.2cm}

Our main theorem is stated as:
\begin{theorem}\label{thm1} Let $n\geq2$ and $M$ be an $n$-dimensional metric
cone, and let $u$ be the solution of the equation \eqref{1.1}.
Suppose $q=\frac{p'(n+2)}{n}>\frac{2(n+1)}{n}$ and $p\geq1$. Then
there exists a constant $C$ only depending on $p,q,n$, and $M$ such
that

$1).$ if $u_0(z)=f(r)$ is a radial Schwartz function\footnote{This
is in order to avoid needless technicalities, but our estimates will
not depend on any of the Schwartz semi-norms of the $u_0$ and so can
be extended to rougher initial data.}, then
\begin{equation}\label{1.5}
\|u(t,z)\|_{L^q_{t,z}(\R\times M)}\leq C_{p,q,n,M}\|\mathcal{F}_{H}
(u_0)\|_{L^p(M)};
\end{equation}

$2).$ and if $u_0$ is any Schwartz function (not necessarily radial)
and $p\geq2$, then
\begin{equation}\label{1.6}
\|u(t,z)\|_{L^q_{t,z}(\R\times M)}\leq C_{p,q,n,M}\|\mathcal{F}_{H}
{\big((1-\widetilde{\Delta}_h)^{s}u_0\big)}\|_{L^p(M)},
\end{equation}
where $s=\frac{(q-2)(n-1)}{4q}+\frac1{qn}$.

\end{theorem}

{\bf Remarks:}\vspace{0.1cm}

$\mathrm{i}).$ We are interested in the estimate \eqref{1.6} with
$p=2$, which gives a global-in-time Strichartz-type estimate with
$s$-loss of angular regularity
$$\|u(t,z)\|_{L^{{2(n+2)}/{n}}_{t,z}(\R\times M)}\leq
C\|(1-\widetilde{\Delta}_h)^su_0\|_{L^2(M)},\quad
s=\frac{(q-2)(n-1)}{4q}+\frac1{qn}.$$ By \eqref{1.5}, we obtain a
global in time Strichartz estimates for radial initial
data.\vspace{0.1cm}

$\mathrm{ii}).$ Let $N$ be a dyadic number, if the initial data
$u_0$ is radial such that the support of
$\mathcal{F}_{H}(u_0)\subset\{\rho:N\leq \rho\leq 2N\}$, by
interpolating \eqref{3.1} and \eqref{3.4} in $q$ and summing in $R$,
we can obtain the Strichartz estimate
\begin{equation}\label{1.7}
\|u(t,z)\|_{L^{q}_{t,z}(\R\times M)}\leq C
N^{\frac{n}2-\frac{n+2}q}\|u_0\|_{L^2(M)}\quad
\text{for}~q>{2(2n+1)}/{(2n-1)}.
\end{equation} The Strichartz
estimates in \cite{HTW1,Miz} also imply \eqref{1.7} holds locally in
time, but for $q\geq2(n+2)/n$.\vspace{0.1cm}

$\mathrm{iii}).$ The assumption on the positivity of the operator
$-\widetilde{\Delta}_h$ can be satisfied when $q(\theta)\geq0$. It
would be possible to generalize the result to
$-\widetilde{\Delta}_h+(n-2)^2/4>0$ allowing some negative
potential, which includes the special Schr\"odinger equation on
$\R^n$ with a inverse-square potential $a/|z|^2$ when
$a>-(n-2)^2/4$. In that case, the relationship between $q$ and $p$
should depend on the square root of the smallest eigenvalue of the
operator $-\widetilde{\Delta}_h+(n-2)^2/4$. \vspace{0.1cm}

$\mathrm{iv}).$ In a future work, we hope to use the resolvent and
spectral measure arguments in \cite{GH,GHS} to show the restriction
estimate for $p=2$ without a loss of angular regularity.
\vspace{0.2cm}

As pointed out in the paper \cite{HTW1}, the Strichartz estimates
established by Hassell, Tao and Wunsch are not strong enough to
obtain a scattering theory for the nonlinear Schr\"odinger equations
on the scattering manifold.  Ford \cite{Ford} proved the
global-in-time Strichartz estimates for two-dimensional metric cone
$C(\mathbb{S}^1_{\rho})$. From Ford's Strichartz estimates, one can
conclude the global existence and scattering for the mass critical
Schr\"odinger equation on $2$-dimension metric cone with small
initial data.  As applications of \eqref{1.5} with $p=2$, we reprove
the same result for the mass critical Schr\"odinger equation on
$2$-dimension metric cone with small radial initial data. We do this
because that one can generalize the result to higher dimension as
long as one could develop a fractional Liebniz rule for Sobolev
spaces on cones. Consider the initial value problem
\begin{equation}\label{1.8}
\begin{cases}
i\partial_t u-H u=\gamma|u|^{2}u,\qquad
(t,z)\in\R\times M, \\
u(t,z)|_{t=0}=u_0(z),\qquad\qquad z\in M.
\end{cases}
\end{equation}
Indeed by duality, the Strichartz estimate \eqref{1.5} implies the
inhomogeneous Strichartz estimate
\begin{equation}
\Big\|\int_0^t
e^{-i(t-s)H}f(z,s)\mathrm{d}s\Big\|_{L^q_{t,z}(\R\times M)}\lesssim
\|f\|_{L^{q'}_{t,z}(\R\times M)}, \quad \text{with}\quad q=2(n+2)/n.
\end{equation}
And then we can apply the arguments of Cazenave and Weissler
\cite{CW} or Tao \cite{Tbook} with Euclidean space replaced by the
conic manifold $M$ to show:
\begin{corollary}[Scattering theory for NLS] Let $M$ be
$2$-dimension manifold as in Theorem \ref{thm1} and $\gamma=\pm1$.
Let $u_0\in L^2(M)$ be radial such that $\|u_0\|_{L^2(M)}\leq
\epsilon$ with small constant $\epsilon$, then NLS \eqref{1.8} is
global well-posed in $L^2(M)$ and the solution $u$ is scattering and
moreover $u\in L^{4}_{t,z}(\R\times M)$.
\end{corollary}

{\bf Remarks:} For higher dimensions $n\geq2$, one could show the
small scattering theory in $H^s(M)$ when $s\geq \text{max}(0,\frac
n2-\frac{2}{\kappa-1})$ for the nonlinear Schr\"odinger equation
\eqref{1.8} with nonlinearity $|u|^{\kappa-1}u, (\kappa>1)$. This
would require one to develop a fractional Liebniz rule for Sobolev
spaces on these manifolds.\vspace{0.2cm}

Now we introduce some notation. We use $A\lesssim B$ to denote
$A\leq CB$ for some large constant C which may vary from line to
line and depend on various parameters, and similarly we use $A\ll B$
to denote $A\leq C^{-1} B$. We employ $A\sim B$ when $A\lesssim
B\lesssim A$. If the constant $C$ depends on a special parameter
other than the above, we shall denote it explicitly by subscripts.
For instance, $C_\epsilon$ should be understood as a positive
constant not only depending on $p, q, n$, and $M$, but also on
$\epsilon$. Throughout this paper, pairs of conjugate indices are
written as $p, p'$, where $\frac{1}p+\frac1{p'}=1$ with $1\leq
p\leq\infty$. We use $L^p_{\mu(r)}(\R_+)$ to denote the usual $L^p$
space with the measure $\mathrm{d}\mu(r)=r^{n-1}\mathrm{d}r$.
\vspace{0.2cm}

This paper is organized as follows: In Section 2, we use the Hankel
transform and Bessel function to give the expression of the
solution. Section 3 is devoted to proving the key localized
estimates of Hankel transforms. In the final section, we use the
estimates established in Section 3 to show Theorem
\ref{thm1}.\vspace{0.2cm}

{\bf Acknowledgments:}\quad The author would like to express his
great gratitude to A. Hassell for his helpful discussions and
comments. He also would like to thank the anonymous referee for
careful reading the manuscript and for giving useful comments. The
author was partly supported by the Fundamental Research Foundation
of Beijing Institute of Technology (20111742015) and Beijing Natural
Science Foundation£¨1144014).

\section{Preliminary}

In this section, we introduce a orthogonal decomposition of
$L^2(\Sigma)$ associated with the eigenfunctions of
$-\Delta_h+q(\theta)$. We provide some standard facts about the
Hankel transform and the Bessel functions. We conclude this section
by writing the solution of \eqref{1.1} as a linear combination of
products of radial functions and the eigenfunctions of
$-\Delta_h+q(\theta)$.\vspace{0.2cm}

\subsection{Orthogonal decomposition of $L^2(\Sigma)$} In this
subsection, we decompose $L^2(\Sigma)$ into the subspaces spanned by
the eigenfunctions of $-\Delta_h+q(\theta)$ associated with its
eigenvalues. We consider the operator
\begin{equation}\label{2.1}
H=-\Delta_{g}+\frac{q(\theta)}{r^2},
\end{equation}
on the metric cone $M=(0,\infty)_r\times\Sigma$. Here $(r,\theta)\in
\R_+\times\Sigma$ are some polar coordinates, $q(\theta)$ is a real
continuous function and the metric  $g$ in coordinates
$(r,\theta)\in \R_+\times\Sigma$ is a metric of the form
\begin{equation*}
g=\mathrm{d}r^2+r^2h(\theta,\mathrm{d}\theta).
\end{equation*}
The Riemannian metric $h$ on $\Sigma$ is independent of $r$. If
$\Sigma$ has a boundary, the Dirichlet condition will be used for
$H$. Let $\Delta_h$ be the Laplace-Beltrami operator on
$(\Sigma,h)$. We will assume that
\begin{equation*}
-\Delta_h+q(\theta)\geq 0
\end{equation*} on $L^2(\Sigma)$,
that is, for any $f\in L^2(\Sigma)$, we have
\begin{equation*}
\left\langle\big(-\Delta_h+q(\theta)\big)f,f\right\rangle_{L^2(\Sigma)}\geq
0 .
\end{equation*}
Then $H\geq0$ in $L^2(M;\mathrm{d}g(z))$ with
$\mathrm{d}g(z)=\sqrt{|g|}\mathrm{d}z$. We modify $\chi_\infty$ by
\begin{equation}\label{2.2}
\chi_\infty=\Big\{\nu: \nu=\sqrt{\lambda+(1/4)(n-2)^2};~
\lambda~\text{is eigenvalue
of}-\widetilde{\Delta}_h:=-\Delta_h+q(\theta)\Big\},
\end{equation}
and let
\begin{equation}\label{2.3}
\chi_K=\chi_{\infty}\cap [0,K],\quad K\in \N.
\end{equation}
For $\nu\in\chi_\infty$, let $d(\nu)$ be the multiplicity of
$\lambda_\nu=\nu^2-\frac14(n-2)^2$ as eigenvalue of
$-\Delta_h+q(\theta)$ and $\{\varphi_{\nu,\ell}(\theta)\}_{1\leq
\ell\leq d(\nu)}$ the eigenfunctions of $-\Delta_h+q(\theta)$, that
is
\begin{equation}\label{2.4}
(-\Delta_h+q(\theta))\varphi_{\nu,\ell}=\lambda_{\nu}\varphi_{\nu,\ell},
\quad \langle
\varphi_{\nu,\ell},\varphi_{\nu,\ell'}\rangle_{L^2(\Sigma)}=\delta_{\ell,\ell'}.
\end{equation}
We remark that $\lambda_\nu\geq0$ hence $\nu\geq (n-2)/2$. Define
\begin{equation*}
\mathcal{H}^{\nu}=\text{span}\{\varphi_{\nu,1},\ldots,
\varphi_{\nu,d(\nu)}\},
\end{equation*}
then we have the orthogonal decomposition
\begin{equation*}
L^2(\Sigma)=\bigoplus_{\nu\in\chi_\infty} \mathcal{H}^{\nu}.
\end{equation*}
Let $\pi_{\nu}$ denote the orthogonal projection:
\begin{equation*}
\pi_{\nu}f=\sum_{\ell=1}^{d(\nu)}\varphi_{\nu,\ell}(\theta)\int_{\Sigma}f(r,\omega)
\varphi_{\nu,\ell}(\omega)\mathrm{d}\sigma_h , \quad f\in L^2(M),
\end{equation*}
where $\mathrm{d}\sigma_h$ is the measure on $\Sigma$ under the
metric $h$. For any $f\in L^2(M)$, we have the expansion formula
\begin{equation}\label{2.5}
f(z)=\sum_{\nu\in\chi_\infty}\pi_{\nu}f
=\sum_{\nu\in\chi_\infty}\sum_{\ell=1}^{d(\nu)}a_{\nu,\ell}(r)\varphi_{\nu,\ell}(\theta)
\end{equation}
where
$a_{\nu,\ell}(r)=\int_{\Sigma}f(r,\theta)\varphi_{\nu,\ell}(\theta)\mathrm{d}\sigma_h$.
By orthogonality, it gives
\begin{equation}\label{2.6}
\|f(z)\|^2_{L^2(\Sigma)}=\sum_{\nu\in\chi_\infty}\sum_{\ell=1}^{d(\nu)}|a_{\nu,\ell}(r)|^2.
\end{equation}
We write $H$ on the cone expressed in polar coordinates as
\begin{equation}\label{2.7}
H=-\partial^2_r-\frac{n-1}r\partial_r+\frac1{r^2}\big(-\Delta_h+q(\theta)\big)
\end{equation}
and set
\begin{equation}\label{2.8}
\begin{split}
A_{\nu}:=-\partial_r^2-\frac{n-1}r\partial_r+\frac{\nu^2-\left(\frac{n-2}2\right)^2}{r^2}
\end{split}
\end{equation}
in $L^2_{\mu(r)}(\R_+)$. In particular, taking $q(\theta)=a\geq0$,
we also can consider the equation \eqref{1.1} perturbed by an
inverse square potential.
\subsection{The Bessel function and Hankel transform} For our purpose, we recall that the Bessel function
$J_{\nu}(r)$ of order $\nu$ is defined by
\begin{equation*}
J_{\nu}(r)=\frac{(r/2)^{\nu}}{\Gamma\left(\nu+\frac12\right)\Gamma(1/2)}\int_{-1}^{1}e^{isr}(1-s^2)^{(2\nu-1)/2}\mathrm{d
}s,
\end{equation*}
where $\nu>-\frac12$ and $r>0$. A simple computation gives the rough
estimates
\begin{equation}\label{2.9}
|J_{\nu}(r)|\leq
\frac{Cr^\nu}{2^\nu\Gamma\left(\nu+\frac12\right)\Gamma(1/2)}\left(1+\frac1{\nu+1/2}\right),
\end{equation}
where $C$ is an absolute constant and the estimate will be mainly
used when $r\lesssim1$. Another well known asymptotic expansion
about the Bessel function is
\begin{equation*}
J_{\nu}(r)=r^{-1/2}\sqrt{\frac2{\pi}}\cos(r-\frac{\nu\pi}2-\frac{\pi}4)+O_{\nu}(r^{-3/2}),\quad
\text{as}~ r\rightarrow\infty
\end{equation*}
but with a constant depending on $\nu$ (see \cite{SW}). As pointed
out in \cite{Stein1}, if one seeks a uniform bound for large $r$ and
$k$, then the best one can do is $|J_{\nu}(r)|\leq C r^{-\frac13}$.
To investigate the behavior of asymptotic on $k$ and $r$, we recall
Schl\"afli's integral representation \cite{Watson} of the Bessel
function: for $r\in\R^+$ and $\nu>-1/2$
\begin{equation}\label{2.10}
\begin{split}
J_{\nu}(r)&=\frac1{2\pi}\int_{-\pi}^\pi
e^{ir\sin\theta-i\nu\theta}\mathrm{d}\theta-\frac{\sin({\nu}\pi)}{\pi}\int_0^\infty
e^{-(r\sinh s+\nu s)}\mathrm{d}s\\&:=\tilde{J}_\nu(r)-E_\nu(r).
\end{split}
\end{equation}
We remark that $E_\nu(r)=0$ when $\nu\in\Z^+$. A simple computation
gives that for $r>0$
\begin{equation}\label{2.11}
|E_\nu(r)|=\Big|\frac{\sin(\nu\pi)}{\pi}\int_0^\infty e^{-(r\sinh
s+\nu s)}\mathrm{d}s\Big|\leq C (r+\nu)^{-1}.
\end{equation}
Next, we recall the properties of Bessel function $J_\nu(r)$ in
\cite{Stein1}, and we refer the readers to \cite{MZZ1} for the
detail proof.
\begin{lemma}[Asymptotics of the Bessel function] \label{Bessel} Assume $\nu\gg1$. Let $J_\nu(r)$ be
the Bessel function of order $\nu$ defined as above. Then there
exist a large constant $C$ and a small constant $c$ independent of
$\nu$ and $r$ such that:

$\bullet$ when $r\leq \frac \nu2$
\begin{equation}\label{2.12}
\begin{split}
|J_\nu(r)|\leq C e^{-c(\nu+r)};
\end{split}
\end{equation}

$\bullet$ when $\frac \nu 2\leq r\leq 2\nu$
\begin{equation}\label{2.13}
\begin{split}
|J_\nu(r)|\leq C \nu^{-\frac13}(\nu^{-\frac13}|r-\nu|+1)^{-\frac14};
\end{split}
\end{equation}

$\bullet$ when $r\geq 2\nu$
\begin{equation}\label{2.14}
\begin{split}
 J_\nu(r)=r^{-\frac12}\sum_{\pm}a_\pm(r,\nu) e^{\pm ir}+E(r,\nu),
\end{split}
\end{equation}
where $|a_\pm(r,\nu)|\leq C$ and $|E(r,\nu)|\leq Cr^{-1}$.
\end{lemma}\vspace{0.2cm}
Let $f\in L^2(M)$, we define the Hankel transform of order $\nu$ by
\begin{equation}\label{2.15}
(\mathcal{H}_{\nu}f)(\rho,\theta)=\int_0^\infty(r\rho)^{-\frac{n-2}2}J_{\nu}(r\rho)f(r,\theta)r^{n-1}\mathrm{d}r.
\end{equation}
As in \cite{BPSS,PSS}, we have the following properties of the
Hankel transform. We also refer the readers to M. Taylor
\cite[Chapter 9]{Taylor}.
\begin{lemma}\label{Hankel}
Let $\mathcal{H}_{\nu}$ and $A_{\nu}$ be defined as above. Then

$(\rm{i})$ $\mathcal{H}_{\nu}=\mathcal{H}_{\nu}^{-1}$,

$(\rm{ii})$ $\mathcal{H}_{\nu}$ is self-adjoint, i.e.
$\mathcal{H}_{\nu}=\mathcal{H}_{\nu}^*$,

$(\rm{iii})$ $\mathcal{H}_{\nu}$ is an $L^2$ isometry, i.e.
$\|\mathcal{H}_{\nu}\phi\|_{L^2(M)}=\|\phi\|_{L^2(M)}$,

$(\rm{iv})$ $\mathcal{H}_{\nu}(
A_{\nu}\phi)(\rho,\theta)=\rho^2(\mathcal{H}_{\nu}
\phi)(\rho,\theta)$, for $\phi\in L^2$.
\end{lemma}


\subsection{The expression of the solution.}
Consider the following Cauchy problem:
\begin{equation}\label{2.16}
\begin{cases}
i\partial_{t}u+H u=0,\\
u(0,z)=u_0(z).
\end{cases}
\end{equation}
By \eqref{2.5}, we have the expansion
\begin{equation*}
u_0(z)=\sum_{\nu\in\chi_\infty}\sum_{\ell=1}^{d(\nu)}a_{\nu,\ell}(r)\varphi_{\nu,\ell}(\theta).
\end{equation*}
Let us consider the equation \eqref{2.16} in polar coordinates
$(r,\theta)$. Write $v(t,r,\theta)=u(t,z)$ and $g(r,\theta)=u_0(z)$.
Then $v(t,r,\theta)$ satisfies that
\begin{equation}\label{2.17}
\begin{cases}
i\partial_{t}
v-\partial_{rr}v-\frac{n-1}r \partial_r v-\frac1{r^2}\Delta_{h}v+\frac{q(\theta)}{r^2}v=0\\
v(0,r,\theta)=g(r,\theta),
\end{cases}
\end{equation}
where
\begin{equation*}
g(r,\theta)=\sum_{\nu\in\chi_\infty}\sum_{\ell=1}^{d(\nu)}a_{\nu,\ell}(r)\varphi_{\nu,\ell}(\theta).
\end{equation*}
Using separation of variables, we can write $v$ as a linear
combination of products of functions  and eigenfunctions
\begin{equation}\label{2.18}
v(t,r,\theta)=\sum_{\nu\in\chi_\infty}\sum_{\ell=1}^{d(\nu)}v_{\nu,\ell}(t,r)\varphi_{\nu,\ell}(\theta),
\end{equation}
where $v_{\nu,\ell}$ is given by
\begin{equation*}
\begin{cases}
i\partial_{t}
v_{\nu,\ell}-\partial_{rr}v_{\nu,\ell}-\frac{n-1}r\partial_rv_{\nu,\ell}+\frac{\lambda_\nu}{r^2}v_{\nu,\ell}=0, \\
v_{\nu,\ell}(0,r)=a_{\nu,\ell}(r)
\end{cases}
\end{equation*}
for each $\nu\in\chi_\infty$ and $1\leq\ell\leq d(\nu)$. Recall
$A_{\nu}$ defined in \eqref{2.8}, then it reduces to consider
\begin{equation}\label{2.19}
\begin{cases}
i\partial_{t}
v_{\nu,\ell}+A_{\nu}v_{\nu,\ell}=0, \\
v_{\nu,\ell}(0,r)=a_{\nu,\ell}(r).
\end{cases}
\end{equation}
Applying the Hankel transform to the equation \eqref{2.19}, we have
by $(\rm{iv})$ in Lemma \ref{Hankel}
\begin{equation}\label{2.20}
\begin{cases}
i\partial_{t}
\tilde{ v}_{\nu,\ell}+\rho^2\tilde{v}_{\nu,\ell}=0 \\
\tilde{v}_{\nu,\ell}(0,\rho)=b_{\nu,\ell}(\rho),
\end{cases}
\end{equation}
where
\begin{equation}\label{2.21}
\tilde{v}_{\nu,\ell}(t,\rho)=(\mathcal{H}_{\nu}
v_{\nu,\ell})(t,\rho),\quad
b_{\nu,\ell}(\rho)=(\mathcal{H}_{\nu}a_{\nu,\ell})(\rho).
\end{equation}
Solving this ODE and inverting the Hankel transform, we obtain
\begin{equation*}
\begin{split}
v_{\nu,\ell}(t,r)&=\int_0^\infty(r\rho)^{-\frac{n-2}2}J_{\nu}(r\rho)\tilde{v}_{\nu,\ell}(t,\rho)\rho^{n-1}\mathrm{d}\rho\\
&=\int_0^\infty(r\rho)^{-\frac{n-2}2}J_{\nu}(r\rho)e^{
it\rho^2}b_{\nu,\ell}(\rho)\rho^{n-1}\mathrm{d}\rho.
\end{split}
\end{equation*}
Therefore we get
\begin{equation}\label{2.22}
\begin{split} &u(t,z)=e^{itH}u_0=v(t,r,\theta)
\\&=\sum_{\nu\in\chi_\infty}\sum_{\ell=1}^{d(\nu)}\varphi_{\nu,\ell}(\theta)\int_0^\infty(r\rho)^{-\frac{n-2}2}J_{\nu}(r\rho)e^{
it\rho^2}b_{\nu,\ell}(\rho)\rho^{n-1}\mathrm{d}\rho
\\&=\sum_{\nu\in\chi_\infty}\sum_{\ell=1}^{d(\nu)}\varphi_{\nu,\ell}(\theta)\mathcal{H}_{\nu}\big[e^{
it\rho^2}b_{\nu,\ell}(\rho)\big](r).
\end{split}
\end{equation}\vspace{0.2cm}

\section{Localized estimates of Hankel transforms}
To prove Theorem \ref{thm1}, we need the following linear localized
estimates. As mentioned in the introduction, we need develop the
decay of the Bessel function and explore the oscillation both in
$e^{it\rho^2}$ and the Bessel function to prove these localized
estimates. Since these estimates take the same form for radial case
and general case, we use the notation $\chi_K$ for finite $K$ or
$K=\infty$ to treat the cases together in the following proof.
\begin{proposition}\label{linear estimates}
Let $\beta \in C_c^\infty(\R)$ supported in $I:=[1,2]$ and $R>0$ be
a dyadic number. Then the following linear restriction estimates
hold:\vspace{0.1cm}

$\bullet$ for $q=2$,
\begin{equation}\label{3.1}
\begin{split}
\big\|\sum_{\nu\in\chi_K}\sum_{\ell=1}^{d(\nu)}\varphi_{\nu,\ell}(\theta)&\mathcal{H}_{\nu}\big[e^{it\rho^2}
b_{\nu,\ell}(\rho)\beta(\rho)\big](r)\big\|_{L^2_t(\R;L^2_{\mu(r)}([R,2R];L^{2}_\theta(\Sigma)))}\\&\lesssim\min\left\{R^\frac12,
R^{\frac
n2}\right\}\Big\|\Big(\sum_{\nu\in\chi_K}\sum_{\ell=1}^{d(\nu)}|b_{\nu,\ell}(\rho)|^2\Big)^{\frac12}\beta(\rho)\Big\|_{L^{2}_{\mu(\rho)}};
\end{split}
\end{equation}

$\bullet$ for $q=\infty$,

\begin{equation}\label{3.2}
\begin{split}
\big\|\sum_{\nu\in\chi_K}\sum_{\ell=1}^{d(\nu)}\varphi_{\nu,\ell}(\theta)&\mathcal{H}_{\nu}\big[e^{it\rho^2}
b_{\nu,\ell}(\rho)\beta(\rho)\big](r)\big\|_{L^\infty_t(\R;L^\infty_{\mu(r)}([R,2R];L^{2}_\theta(\Sigma)))}\\&\lesssim
\min\left\{R^{-\frac{n-1}2},
1\right\}\Big\|\Big(\sum_{\nu\in\chi_K}\sum_{\ell=1}^{d(\nu)}(1+\nu)^{\frac13}|b_{\nu,\ell}(\rho)|^2\Big)^{\frac12}\beta(\rho)\Big\|_{L^{1}_{\mu(\rho)}};
\end{split}
\end{equation}
and
\begin{equation}\label{3.3}
\begin{split}
\big\|\sum_{\nu\in\chi_K}\sum_{\ell=1}^{d(\nu)}\varphi_{\nu,\ell}(\theta)&\mathcal{H}_{\nu}\big[e^{it\rho^2}
b_{\nu,\ell}(\rho)\beta(\rho)\big](r)\big\|_{L^\infty_t(\R;L^\infty_{\mu(r)}([R,2R];L^{2}_\theta(\Sigma)))}\\&\lesssim
\min\left\{R^{-\frac{n-1}2},
1\right\}\Big\|\Big(\sum_{\nu\in\chi_K}\sum_{\ell=1}^{d(\nu)}|b_{\nu,\ell}(\rho)|^2\Big)^{\frac12}\beta(\rho)\Big\|_{L^{2}_{\mu(\rho)}};
\end{split}
\end{equation}

$\bullet$ for $q=3p'$ and $2\leq p<4$,
\begin{equation}\label{3.4}
\begin{split}
\big\|\sum_{\nu\in\chi_K}\sum_{\ell=1}^{d(\nu)}&\varphi_{\nu,\ell}(\theta)\mathcal{H}_{\nu}\big[e^{it\rho^2}
b_{\nu,\ell}(\rho)\beta(\rho)\big](r)\big\|_{L^q_t(\R;L^q_{\mu(r)}([R,2R];L^{2}_\theta(\Sigma)))}\\&\lesssim
\min\left\{R^{(n-1)(\frac1q-\frac12)}, R^{\frac
nq}\right\}\Big\|\Big(\sum_{\nu\in\chi_K}\sum_{\ell=1}^{d(\nu)}(1+\nu)^{\frac4q}|b_{\nu,\ell}(\rho)|^2\Big)^{\frac12}\beta(\rho)\Big\|_{L^{p}_{\mu(\rho)}};
\end{split}
\end{equation}
and $1\leq p <2$
\begin{equation}\label{3.5}
\begin{split}
\big\|\sum_{\nu\in\chi_K}\sum_{\ell=1}^{d(\nu)}&\varphi_{\nu,\ell}(\theta)\mathcal{H}_{\nu}\big[e^{it\rho^2}
b_{\nu,\ell}(\rho)\beta(\rho)\big](r)\big\|_{L^q_t(\R;L^q_{\mu(r)}([R,2R];L^{2}_\theta(\Sigma)))}\\&\lesssim
\min\left\{R^{(n-1)(\frac1q-\frac12)}, R^{\frac
nq}\right\}\Big\|\Big(\sum_{\nu\in\chi_K}\sum_{\ell=1}^{d(\nu)}(1+\nu)^{\frac2q+\frac13}|b_{\nu,\ell}(\rho)|^2\Big)^{\frac12}\beta(\rho)\Big\|_{L^{p}_{\mu(\rho)}};
\end{split}
\end{equation}

$\bullet$ for $q=4$ and $\forall \epsilon>0$
\begin{equation}\label{3.6}
\begin{split}
\big\|\sum_{\nu\in\chi_K}\sum_{\ell=1}^{d(\nu)}&\varphi_{\nu,\ell}(\theta)\mathcal{H}_{\nu}\big[e^{it\rho^2}
b_{\nu,\ell}(\rho)\beta(\rho)\big](r)\big\|_{L^4_t(\R;L^4_{\mu(r)}([R,2R];L^{2}_\theta(\Sigma)))}\\&\lesssim
\min\left\{R^{-\frac{n-1}4+\epsilon}, R^{\frac
n4}\right\}\Big\|\Big(\sum_{\nu\in\chi_K}\sum_{\ell=1}^{d(\nu)}(1+\nu)|b_{\nu,\ell}(\rho)|^2\Big)^{\frac12}\beta(\rho)\Big\|_{L^{4}_{\mu(\rho)}}.
\end{split}
\end{equation}
\end{proposition}

\begin{remark}
The estimates above are essentially established by breaking things
into $R\lesssim1$ and $R\gg1$ due to the different asymptotic
behavior of Bessel function on each regime.

\end{remark}

\begin{remark}\label{K} The implicit constant is independent of $K$, which allows us to sum over all of $\chi_\infty$ in next section.
In other words, we can replace $\chi_K$ by $\chi_\infty$ in the
above estimates. When the initial data is radial Schwartz function,
$K$ is finite hence the sum over $\ell$ and $\nu$ converges. If the
initial data is a Schwartz function (not necessary radial), $K$ may
be infinite, and however the summation also converges due to the
Schwartz property. More precisely, since the initial data is
Schwartz, $b_{\nu,\ell}$ decays likely $(1+\nu)^{-N}$ for any $N>0$.
On the other hand, we note $d(\nu)\sim \nu^{n-2}$ hence the sum
converges.
\end{remark}

\begin{remark}The loss of angular regularity in \eqref{3.5} is
much more than \eqref{3.4}. We only use \eqref{3.5} to conclude
\eqref{1.5}. By the radial assumption, one has that $K$ is finite
hence the loss of angular regularity is trivial.
\end{remark}

The rest of this section is devoted to proving this Proposition. We
first note that by orthogonality of the angular eigenfunctions
$\varphi_{\nu,\ell}$
\begin{equation}\label{new}
\begin{split}
\big\|\sum_{\nu\in\chi_K}\sum_{\ell=1}^{d(\nu)}\varphi_{\nu,\ell}(\theta)&\mathcal{H}_{\nu}\big[e^{it\rho^2}
b_{\nu,\ell}(\rho)\beta(\rho)\big](r)\big\|_{L^{2}_\theta(\Sigma)}\\&=\left\{
\sum_{\nu\in\chi_K}\sum_{\ell=1}^{d(\nu)}\Big|\mathcal{H}_{\nu}\big[e^{it\rho^2}
b_{\nu,\ell}(\rho)\beta(\rho)\big](r)\Big|^2\right\}^{1/2}\\&=r^{-\frac{n-2}2}\left(\sum_{\nu\in\chi_K}\sum_{\ell=1}^{d(\nu)}\big|\int_0^\infty
e^{it\rho^2} J_{\nu}( r\rho)b_{\nu,\ell}(\rho)\beta(\rho)\rho^{
n/2}\mathrm{d}\rho\big|^2\right)^{\frac12}.
\end{split}
\end{equation}
Now we prove \eqref{3.1}-\eqref{3.6} hold for $ R\lesssim1$. To do
this, we need the following Lemma.

\begin{lemma}\label{lem1} Let $b_{\nu,\ell}(\rho)$ and $\beta(\rho)$ be as
in Proposition \ref{linear estimates}, then the following estimate
holds for $q\geq2$ and $R\lesssim1$
\begin{equation}\label{3.7}
\begin{split}
\Big\|r^{-\frac{n-2}2}\Big(\sum_{\nu\in\chi_K}\sum_{\ell=1}^{d(\nu)}\big|\int_0^\infty
e^{it\rho^2} J_{\nu}(
r\rho)&b_{\nu,\ell}(\rho)\beta(\rho)\rho^{\frac
n2}\mathrm{d}\rho\big|^2\Big)^{\frac12}
\Big\|_{L^q_t(\R;L^q_{\mu(r)}[R,2R])}\\&\lesssim R^{\frac nq}
\Big\|\Big(\sum_{\nu\in\chi_K}\sum_{\ell=1}^{d(\nu)}|b_{\nu,\ell}(\rho)|^2\Big)^{\frac12}\beta(\rho)\Big\|_{L^{q'}_{\mu(\rho)}}.
\end{split}
\end{equation}
\end{lemma}
We postpone the proof for a moment. Notice the $\nu$-weights
appearing in \eqref{3.2}, \eqref{3.4}-\eqref{3.6} are larger than
$1$,  and note $q'\leq p$ and compact support of $\beta$, we use the
H\"older inequality and Lemma \ref{lem1} to show Proposition
\ref{linear estimates} holds for $R\lesssim 1$.
\begin{proof}[Proof of Lemma~\ref{lem1}] Since $q\geq 2$, the Minkowski
inequality and Fubini's theorem show that the left hand side of
\eqref{3.7} is bounded by
\begin{equation*}
\begin{split}
\left\|r^{-\frac{n-2}2}\left(\sum_{\nu\in\chi_K}\sum_{\ell=1}^{d(\nu)}\Big\|\int_0^\infty
e^{it\rho^2} J_{\nu}( r\rho)b_{\nu,\ell}(\rho)\beta(\rho)\rho^{\frac
{n-2}2}\rho\mathrm{d}\rho\Big\|^2_{L^q_t(\R)}\right)^{\frac12}\right\|_{L^q_{\mu(r)}([R,2R])}.
\end{split}
\end{equation*}
We write by making variable changes
\begin{equation}\label{var}
\begin{split}
\left\|r^{-\frac{n-2}2}\left(\sum_{\nu\in\chi_K}\sum_{\ell=1}^{d(\nu)}\Big\|\int_0^\infty
e^{it\rho} J_{\nu}(
r\sqrt{\rho})b_{\nu,\ell}(\sqrt{\rho})\beta(\sqrt{\rho})\rho^{\frac
{n-2}4}\mathrm{d}\rho\Big\|^2_{L^q_t(\R)}\right)^{\frac12}\right\|_{L^q_{\mu(r)}([R,2R])}.
\end{split}
\end{equation}
Hence we use the Hausdorff-Young inequality in $t$ and change
variables back to obtain
\begin{equation*}
\begin{split}
\text{LHS of }~\eqref{3.7}\lesssim
\Big\|r^{-\frac{n-2}2}\Big(\sum_{\nu\in\chi_K}\sum_{\ell=1}^{d(\nu)}
\big\|J_{\nu}(
r\rho)b_{\nu,\ell}(\rho)\beta(\rho)\rho^{(n-2)/2+1/q'}\big\|_{L^{q'}_\rho}^2\Big)^{\frac12}
\Big\|_{L^q_{\mu(r)}([R,2R])}.
\end{split}
\end{equation*}
Note the compact support of $\beta$, we obtain by \eqref{2.9}
\begin{equation*}
\begin{split}
&\text{LHS of}~\eqref{3.7} \\ \lesssim &\left(\int_{R}^{2R}
r^{-\frac{(n-2)q}2}\Big(\sum_{\nu\in\chi_K}\sum_{\ell=1}^{d(\nu)}
\Big|\frac{(4
r)^{\nu}}{2^{\nu}\Gamma(\nu+\frac12)\Gamma(1/2)}\Big|^2\big\|b_{\nu,\ell}(\rho)\beta(\rho)\big\|_{L^{q'}_\rho}^2\Big)^{\frac
q2} r^{n-1}\mathrm{d}r\right)^{\frac1q}.
\end{split}
\end{equation*}
Note the stirling's formula $\Gamma\left(\nu+1\right)\sim
\sqrt{\nu}(\nu/e)^\nu$, we see the coefficient is bounded
independent of $\nu$. On the other hand, we have the factor
$R^{n/q}R^{\sqrt{\lambda_0+(n-2)^2/4}-(n-2)/2}$ where
$\lambda_0\geq0$ is the smallest eigenvalue of
$-\Delta_h+q(\theta)$. Note compact support of $\beta$, thus we can
adjust the weight in $\rho$ to prove \eqref{3.7}.
\end{proof}

\begin{remark} It might help to given an example to show how this
works. If $h=(d\theta)^2$ is the Euclidean metric on the sphere
$\mathbb{S}^{n-1}$, $n\geq 2$ and $q(\theta)=0$, then we have for
$H=-\Delta$
$$\chi_\infty=\left\{(n-2)/2+k; k\in\mathbb{N}\right\}.$$
One can follow the above argument to show \eqref{3.7}.
\end{remark}

To prove Proposition \ref{linear estimates}, it suffices to prove
the followings estimates: for $R\gg1$

$\bullet$ for $q=2$
\begin{equation}\label{3.8}
\begin{split}
\Big\|r^{-\frac{n-2}2}\Big(\sum_{\nu\in\chi_K}\sum_{\ell=1}^{d(\nu)}\big|\int_0^\infty
e^{it\rho^2} J_{\nu}(
r\rho)&b_{\nu,\ell}(\rho)\beta(\rho)\rho^{\frac
n2}\mathrm{d}\rho\big|^2\Big)^{\frac12}
\Big\|_{L^2_t(\R;L^2_{\mu(r)}[R,2R])}\\&\lesssim R^\frac12
\big(\sum_{\nu\in\chi_K}\sum_{\ell=1}^{d(\nu)}\|b_{\nu,\ell}(\rho)\beta(\rho)\|^2_{L^2_{\mu(\rho)}}\big)^{\frac12};
\end{split}
\end{equation}

$\bullet$ for $q=\infty$
\begin{equation}\label{3.9}
\begin{split}
\Big\|r^{-\frac{n-2}2}\Big(\sum_{\nu\in\chi_K}\sum_{\ell=1}^{d(\nu)}\big|\int_0^\infty
&e^{it\rho^2} J_{\nu}(
r\rho)b_{\nu,\ell}(\rho)\beta(\rho)\rho^{\frac
n2}\mathrm{d}\rho\big|^2\Big)^{\frac12}
\Big\|_{L^\infty_t(\R;L^\infty_{\mu(r)}[R,2R])}\\&\lesssim
R^{-\frac{n-1}2}\Big\|\big(\sum_{\nu\in\chi_K}\sum_{\ell=1}^{d(\nu)}(1+\nu)^{\frac13}
\big|b_{\nu,\ell}(\rho)\big|^2\big)^{\frac
12}\beta(\rho)\Big\|_{L^{1}_{\mu(\rho)}};
\end{split}
\end{equation}
and
\begin{equation}\label{3.10}
\begin{split}
\Big\|r^{-\frac{n-2}2}\Big(\sum_{\nu\in\chi_K}\sum_{\ell=1}^{d(\nu)}\big|\int_0^\infty
e^{it\rho^2} J_{\nu}(
r\rho)&b_{\nu,\ell}(\rho)\beta(\rho)\rho^{\frac
n2}\mathrm{d}\rho\big|^2\Big)^{\frac12}
\Big\|_{L^\infty_t(\R;L^\infty_{\mu(r)}[R,2R])}\\&\lesssim
R^{-\frac{n-1}2}\big(\sum_{\nu\in\chi_K}\sum_{\ell=1}^{d(\nu)}\|b_{\nu,\ell}(\rho)\beta(\rho)\|^2_{L^2_{\mu(\rho)}}\big)^{\frac12};
\end{split}
\end{equation}

$\bullet$ for $q=3p'$ and $2\leq p<4$,
\begin{equation}\label{3.11}
\begin{split}
\Big\|r^{-\frac{n-2}2}\Big(\sum_{\nu\in\chi_K}\sum_{\ell=1}^{d(\nu)}\big|\int_0^\infty
&e^{it\rho^2} J_{\nu}(
r\rho)b_{\nu,\ell}(\rho)\beta(\rho)\rho^{\frac
n2}\mathrm{d}\rho\big|^2\Big)^{\frac12}
\Big\|_{L^q_t(\R;L^q_{\mu(r)}[R,2R])}\\&\lesssim
R^{(n-1)(\frac{1}q-\frac12)}\Big\|\big(\sum_{\nu\in\chi_K}\sum_{\ell=1}^{d(\nu)}
(1+\nu)^{\frac{4}{q}}\big|b_{\nu,\ell}(\rho)\big|^2\big)^{\frac
12}\beta(\rho)\Big\|_{L^{p}_{\mu(\rho)}};
\end{split}
\end{equation}
and $1\leq p<2$
\begin{equation}\label{3.12}
\begin{split}
\Big\|r^{-\frac{n-2}2}\Big(\sum_{\nu\in\chi_K}\sum_{\ell=1}^{d(\nu)}\big|\int_0^\infty
&e^{it\rho^2} J_{\nu}(
r\rho)b_{\nu,\ell}(\rho)\beta(\rho)\rho^{\frac
n2}\mathrm{d}\rho\big|^2\Big)^{\frac12}
\Big\|_{L^q_t(\R;L^q_{\mu(r)}[R,2R])}\\&\lesssim
R^{(n-1)(\frac{1}q-\frac12)}\Big\|\big(\sum_{\nu\in\chi_K}\sum_{\ell=1}^{d(\nu)}
(1+\nu)^{{\frac{2}{q}+\frac13}}\big|b_{\nu,\ell}(\rho)\big|^2\big)^{\frac
12}\beta(\rho)\Big\|_{L^{p}_{\mu(\rho)}};
\end{split}
\end{equation}

$\bullet$ for $q=4$, $\forall \epsilon>0$
\begin{equation}\label{3.13}
\begin{split}
\Big\|r^{-\frac{n-2}2}\Big(\sum_{\nu\in\chi_K}\sum_{\ell=1}^{d(\nu)}\big|\int_0^\infty
&e^{it\rho^2} J_{\nu}(
r\rho)b_{\nu,\ell}(\rho)\beta(\rho)\rho^{\frac
n2}\mathrm{d}\rho\big|^2\Big)^{\frac12}
\Big\|_{L^4_t(\R;L^4_{\mu(r)}[R,2R])}\\&\lesssim
R^{-\frac{n-1}4+\epsilon}\Big\|\big(\sum_{\nu\in\chi_K}\sum_{\ell=1}^{d(\nu)}
(1+\nu)\big|b_{\nu,\ell}(\rho)\big|^2\big)^{\frac
12}\beta(\rho)\Big\|_{L^{4}_{\mu(\rho)}}.
\end{split}
\end{equation}

{\bf Step 1.} We first prove \eqref{3.8} holds for $R\gg1$. After
changing variables as \eqref{var} and canceling some factors $r$, we
use the Plancherel theorem in $t$ to show
\begin{equation}\label{3.14}
\begin{split}
\text{LHS of }~\eqref{3.8}\lesssim R^{\frac12}\Big\|
\left(\sum_{\nu\in\chi_K}\sum_{\ell=1}^{d(\nu)}\big\|J_{\nu}( r\rho)
b_{\nu,\ell}(\rho)
\beta(\rho)\rho^{(n-1)/2}\big\|^2_{L^{2}_\rho}\right)^{\frac12}
\Big\|_{L^2_{r}([R,2R])}.
\end{split}
\end{equation}
Along with \eqref{3.14}, it is easy to verify \eqref{3.8}, if we
could prove
\begin{equation}\label{3.15}
\int_R^{2R}|J_{\nu}(r)|^2\mathrm{d}r\leq C,\quad R\gg 1,
\end{equation}
where the constant $C$ is independent of $\nu$ and $R$. To prove
\eqref{3.15}, we write
\begin{equation*}
\begin{split}
\int_R^{2R}|J_{\nu}(r)|^2\mathrm{d}r=\int_{I_1}|J_{\nu}(r)|^2\mathrm{d}r
+\int_{I_2}|J_{\nu}(r)|^2\mathrm{d}r+\int_{I_3}|J_{\nu}(r)|^2\mathrm{d}r
\end{split}
\end{equation*}
where $I_1=[R,2R]\cap[0,\frac \nu 2], I_2=[R,2R]\cap[\frac \nu
2,2\nu]$ and $I_3=[R,2R]\cap[2\nu,\infty]$. By using \eqref{2.12}
and \eqref{2.14} in Lemma \ref{Bessel}, we have
\begin{equation}\label{3.16}
\begin{split}
\int_{I_1}|J_{\nu}(r)|^2\mathrm{d}r\leq C
\int_{I_1}e^{-cr}\mathrm{d}r\leq C e^{-cR},
\end{split}
\end{equation}
and \begin{equation}\label{3.17}
\begin{split}\int_{I_3}|J_{\nu}(r)|^2\mathrm{d}r\leq C.\end{split}
\end{equation}
On the other hand, one has by \eqref{2.13}
\begin{equation*}
\begin{split}
\int_{[\frac \nu 2,2\nu]}|J_{\nu}(r)|^2\mathrm{d}r&\leq C
\int_{[\frac \nu
2,2\nu]}\nu^{-\frac23}(1+\nu^{-\frac13}|r-\nu|)^{-\frac
12}\mathrm{d}r\leq C.
\end{split}
\end{equation*}
Observing $[R,2R]\cap[\frac \nu 2,2\nu]=\emptyset$ unless $R\sim
\nu$, we obtain
\begin{equation}\label{3.18}
\begin{split}
\int_{I_2}|J_{\nu}(r)|^2\mathrm{d}r\leq C.
\end{split}
\end{equation}
This together with \eqref{3.16} and \eqref{3.17} yields
\eqref{3.15}. Hence we finally prove \eqref{3.8}. \vspace{0.2cm}

{\bf Step 2.} To prove \eqref{3.9} and \eqref{3.10} hold for
$R\gg1$, we utilize the Schl\"afli's integral representation of the
Bessel function \eqref{2.10} to write ${J}_{\nu}( r\rho)={E}_{\nu}(
r\rho)+\tilde{J}_{\nu}(r\rho)$.  As before using the Minkowski
inequality and the Hausdorff-Young inequality in $t$, we have by
\eqref{2.11},
\begin{equation*}
\begin{split}
\Big\|r^{-\frac{n-2}2}\Big(\sum_{\nu\in\chi_K}\sum_{\ell=1}^{d(\nu)}\big|\int_0^\infty
e^{it\rho^2}
&{E}_{\nu}(r\rho)b_{\nu,\ell}(\rho)\beta(\rho)\rho^{\frac
n2}\mathrm{d}\rho\big|^2\Big)^{\frac12}
\Big\|_{L^\infty_t(\R;L^\infty_{\mu(r)}([R,2R]))}\\&\lesssim
R^{-\frac n2}\Big\|\big(\sum_{\nu\in\chi_K}\sum_{\ell=1}^{d(\nu)}
\big|b_{\nu,\ell}(\rho)\big|^2\big)^{\frac
12}\beta(\rho)\Big\|_{L^{1}_{\mu(\rho)}}.
\end{split}
\end{equation*}
\vspace{0.2cm} Thus it remains to prove \eqref{3.9} and \eqref{3.10}
replacing ${J}_{\nu}$ by $\tilde{J}_{\nu}$. We decompose $[-\pi,
\pi]$ into three partitions as follows $$[-\pi, \pi]=I_1\cup I_2\cup
I_3$$ where
\begin{equation}\label{3.19}
I_1=\{\theta:|\theta|\leq \delta\},\quad I_2=[-\pi,
-\frac\pi2-\delta]\cup[\frac\pi2+\delta, \pi], \quad
I_3=[-\pi,\pi]\setminus(I_1\cup I_2),
\end{equation} with $0<\delta\ll1$. We define
\begin{equation}\label{3.20}
\Phi_{r,\nu}(\theta)=\sin\theta- \nu\theta/r,
\end{equation} and $\chi_\delta(\theta)$ is a smooth function given by
\begin{equation*}
\chi_{\delta}(\theta)=
\begin{cases}1,\quad
\theta\in[-\delta,\delta];\\0, \quad\theta\not\in[-2\delta,2\delta].
\end{cases}
\end{equation*}
Then we divide $\tilde{J}_\nu(r)$ into three pieces and write
\begin{equation}\label{3.21}
\begin{split}
\tilde{J}_\nu(r)&=\frac1{2\pi}\int_{-\pi}^\pi
e^{ir\Phi_{r,\nu}(\theta)}\mathrm{d}\theta\\&=\frac1{2\pi}\Big(\int_{-\pi}^{\pi}
e^{ir\Phi_{r,\nu}(\theta)}\chi_{\delta}(\theta)\mathrm{d}\theta+\int_{I_2}
e^{ir\Phi_{r,\nu}(\theta)}\mathrm{d}\theta+\int_{I_3}
e^{ir\Phi_{r,\nu}(\theta)}(1-\chi_{\delta}(\theta))\mathrm{d}\theta\Big)\\&=:\tilde{J}^1_\nu(r)+\tilde{J}^2_\nu(r)+\tilde{J}^3_\nu(r).
\end{split}
\end{equation}
When $\theta\in I_2$, the function
$\Phi'_{r,\nu}(\theta)=\cos\theta-\nu/r$ is monotonic in the
intervals $[-\pi, -\frac\pi2-\delta]$ and $[\frac\pi2+\delta, \pi]$
respectively and satisfies that
\begin{equation*}
|\Phi'_{r,\nu}(\theta)|\geq \nu/r+|\cos\theta|\geq\sin\delta.
\end{equation*}
Then by \cite [Proposition 2, Chapter VIII]{Stein1}, we have the
following estimate uniformly in $\nu$
\begin{equation}\label{3.22}
\Big|\frac1{2\pi}\int_{I_2}
e^{ir\Phi_{r,\nu}(\theta)}\mathrm{d}\theta\Big|\leq c_\delta r^{-1}.
\end{equation}
When $\theta\in I_3$, then $|\Phi''_{r,\nu}(\theta)|\geq
\sin\delta$, we have by \cite [Proposition 2, Chapter VIII]{Stein1}
\begin{equation}\label{3.23}
\Big|\frac1{2\pi}\int_{I_3}
e^{ir\Phi_{r,\nu}(\theta)}(1-\chi_{\delta}(\theta))\mathrm{d}\theta\Big|\leq
c_\delta r^{-1/2},
\end{equation}
uniformly in $\nu$.  Using the similar arguments as  above, it
follows from \eqref{3.22} and \eqref{3.23} that
\begin{equation}\label{3.24}
\begin{split}
\Big\|r^{-\frac{n-2}2}\Big(\sum_{\nu\in\chi_K}\sum_{\ell=1}^{d(\nu)}\big|\int_0^\infty
e^{it\rho^2}& \big({\tilde{J}^2}_{\nu}( r\rho)+{\tilde{J}^3}_{\nu}(
r\rho)\big)b_{\nu,\ell}(\rho)\beta(\rho)\rho^{\frac
n2}\mathrm{d}\rho\big|^2\Big)^{\frac12}
\Big\|_{L^\infty_t(\R;L^\infty_{\mu(r)}([R,2R]))}\\&\lesssim
R^{-\frac{n-1}2}\Big\|\big(\sum_{\nu\in\chi_K}\sum_{\ell=1}^{d(\nu)}
\big|b_{\nu,\ell}(\rho)\big|^2\big)^{\frac
12}\beta(\rho)\Big\|_{L^{1}_{\mu(\rho)}}.
\end{split}
\end{equation}
By using Lemma \ref{Bessel}, we see $|\tilde{J}^1_{\nu}( r)|\lesssim
r^{-1/3}$ when $r\sim\nu$. Then arguing as before, we have
\begin{equation}\label{3.25}
\begin{split}
\Big\|r^{-\frac{n-2}2}\Big(\sum_{\nu\in\chi_K}\sum_{\ell=1}^{d(\nu)}\big|\int_0^\infty
e^{it\rho^2}& {\tilde{J}^1}_{\nu}(
r\rho)b_{\nu,\ell}(\rho)\beta(\rho)\rho^{\frac
n2}\mathrm{d}\rho\big|^2\Big)^{\frac12}
\Big\|_{L^\infty_t(\R;L^\infty_{\mu(r)}([R,2R]))}\\&\lesssim
R^{-\frac{n-1}2}\Big\|\big(\sum_{\nu\in\chi_K}\sum_{\ell=1}^{d(\nu)}(1+\nu)^{\frac13}
\big|b_{\nu,\ell}(\rho)\big|^2\big)^{\frac
12}\beta(\rho)\Big\|_{L^{1}_{\mu(\rho)}}.
\end{split}
\end{equation}
We here obtain more decay $r^{-\frac16}$ from the loss of the
angular regularity $\nu^{1/6}$ when $r\sim \nu$. Therefore we prove
\eqref{3.9}. To prove \eqref{3.10} concerning
$\tilde{J}^1_{\nu}(\rho r)$ without loss of angular regularity, we
need to use effectively the oscillation of $e^{ it\rho^2}$. We write
Fourier series of $b_{\nu,\ell}(\rho)$ as
\begin{equation*}
\begin{split}
b_{\nu,\ell}(\rho)=\sum_j b_{\nu,\ell}^j e^{i\rho^2 j}\quad
\text{with}\quad b_{\nu,\ell}^j
=\frac{1}{4\pi}\int_0^{16}e^{-i\rho^2
j}b_{\nu,\ell}(\rho)\rho\mathrm{d}\rho.
\end{split}
\end{equation*}
By the Plancherel theorem and the orthogonality, we remark that
\begin{equation}\label{3.26}
\begin{split}
\big(\sum_{\nu\in\chi_K}\sum_{\ell=1}^{d(\nu)}\|b_{\nu,\ell}(\rho)\rho^{\frac12}\|^2_{L^2_{\mu(\rho)}(I)}\big)^{\frac12}
\cong\big(\sum_{\nu\in\chi_K}\sum_{\ell=1}^{d(\nu)}\sum_j|b_{\nu,\ell}^j|^2\big)^{\frac12}.
\end{split}
\end{equation}
Thus it suffices to prove
\begin{equation}\label{3.27}
\begin{split}
\Big\|r^{-\frac{n-2}2}\Big(\sum_{\nu\in\chi_K}\sum_{\ell=1}^{d(\nu)}\big|\int_0^\infty
e^{it\rho^2}& {\tilde{J}^1}_{\nu}(r\rho)\sum_j b_{\nu,\ell}^j
e^{i\rho^2 j}\beta(\rho)\rho^{\frac
n2}\mathrm{d}\rho\big|^2\Big)^{\frac12}
\Big\|_{L^\infty_t(\R;L^\infty_{\mu(r)}([R,2R]))}\\&\lesssim
R^{-\frac{n-1}2}\Big\|\big(\sum_{\nu\in\chi_K}\sum_{\ell=1}^{d(\nu)}
\big|b_{\nu,\ell}(\rho)\big|^2\big)^{\frac
12}\beta(\rho)\Big\|_{L^{2}_{\mu(\rho)}}.
\end{split}
\end{equation}
For simplicity, we define
\begin{equation}\label{3.28}
\begin{split}
\psi_{t+\frac j4}^\nu(r)=\int_0^\infty e^{i(t+\frac j4)\rho^2}
\int_{\R} e^{ i\rho r\sin\theta-
i\nu\theta}\chi_{\delta}(\theta)\mathrm{d}\theta\beta(\rho)\rho^{\frac
n2}\mathrm{d}\rho.
\end{split}
\end{equation}
Let $m=t+\frac j4$, then we write
\begin{equation}\label{3.29}
\begin{split}
\psi_{m}^\nu(r)&=\int_{\R^2} e^{ i\rho(r\sin\theta+\rho m)} e^{-
i\nu\theta}\beta(\rho)\mathrm{d}\rho\chi_{\delta}(\theta)\mathrm{d}\theta.
\end{split}
\end{equation}
For our purpose, we need to investigate the asymptotic behavior of
the function $\psi_{m}^\nu(r)$. To this end, we consider the
following two cases. Write the phase function
\begin{equation*}
\begin{split}
\Phi_{r,m,\nu}(\rho,\theta)=m\rho^2+\rho r\sin\theta-\nu\theta.
\end{split}
\end{equation*}

$\bullet$ Subcase $(a)$: $4R\leq |m|$. Since $R\geq 1$, then
$|m|\geq 4$. Note that $\rho\in[1/2,4]$, then the derivative of the
phase function in $\rho$ satisfies
\begin{equation*}
\begin{split}
|\partial_\rho\Phi_{r,m,\nu}(\rho,\theta)|=|r\sin\theta+2m\rho|\geq
|m|-r|\sin\theta|\geq |m|/100,
\end{split}
\end{equation*}
by making use of $r\leq 2R\leq |m|$ and $|\theta|\leq 2\delta$.
Integrating by part in $\rho$ gives that
\begin{equation}\label{3.30}
\begin{split}
|\psi_{m}^\nu(r)|\leq C_{\delta,N}(1+|m|)^{-N}.
\end{split}
\end{equation}
Hence keeping in mind $m=t+\frac j4$, we have
\begin{equation*}
\begin{split}
&\Big\|r^{-\frac{n-2}2}\Big(\sum_{\nu\in\chi_K}\sum_{\ell=1}^{d(\nu)}\big|\sum_{\{j:
4R\leq|t+\frac j4|\}} b_{\nu,\ell}^j \psi_{t+\frac
j4}^\nu(r)\big|^2\Big)^{\frac12}
\Big\|_{L^\infty_t(\R;L^\infty_{\mu(r)}([R,2R]))}\\&\leq
C_{\delta,N}R^{-N}\Big\|r^{-\frac{n-2}2}\Big(\sum_{\nu\in\chi_K}\sum_{\ell=1}^{d(\nu)}\Big|\sum_{\{j:
4R\leq|t+\frac j4|\}} |b_{\nu,\ell}^j|\left(1+|t+\frac
j4|\right)^{-N} \Big|^2\Big)^{\frac12}
\Big\|_{L^\infty_t(\R;L^\infty_{\mu(r)}([R,2R]))}.
\end{split}
\end{equation*}
By the Cauchy-Schwarz inequality and choosing $N$ large enough, the
above is bounded by
\begin{equation}\label{3.31}
\begin{split}
&C_{\delta,N}R^{-N}\Big\|r^{-\frac{n-2}2}\Big(\sum_{\nu\in\chi_K}\sum_{\ell=1}^{d(\nu)}\sum_{j}
|b_{\nu,\ell}^j|^2(1+|t+\frac j4|)^{-N} \Big)^{\frac12}
\Big\|_{L^\infty_t(\R;L^\infty_{\mu(r)}([R,2R]))}\\&\leq
C_{\delta,N}R^{-N}\Big(\sum_{\nu\in\chi_K}\sum_{\ell=1}^{d(\nu)}\sum_{j}
|b_{\nu,\ell}^j|^2 \Big)^{\frac12}\lesssim
R^{-N}\Big\|\big(\sum_{\nu\in\chi_K}\sum_{\ell=1}^{d(\nu)}
\big|b_{\nu,\ell}(\rho)\big|^2\big)^{\frac
12}\beta(\rho)\Big\|_{L^{2}_{\mu(\rho)}(I)}.
\end{split}
\end{equation}

$\bullet$  Subcase $(b)$: $|m|<4R$. We recall that
\begin{equation*}
\begin{split}
\psi_{m}^\nu(r)&=\frac1{2\pi}\int_{\R^2}
e^{ir\tilde{\Phi}_{r,m,\nu}(\rho,\theta)}\beta(\rho)\chi_{\delta}(\theta)\mathrm{d}\rho\mathrm{d}\theta,
\end{split}
\end{equation*}
where
$\tilde{\Phi}_{r,m,\nu}(\rho,\theta)=\Phi_{r,m,\nu}(\rho,\theta)/r$.
Then a direct computation yields
\begin{equation}\label{3.32}
\begin{split}
\nabla_{\rho,\theta}\tilde{\Phi}_{r,m,\nu}=\Big(-2m\rho/r+
\sin\theta, \rho \cos\theta-\nu/r\Big)
\end{split}
\end{equation}
and
\begin{equation}\label{3.33}
\begin{split}
\frac{\partial^2\tilde{\Phi}_{r,m,\nu}}{\partial(\rho,\theta)^2}=\begin{pmatrix}
-2 m/r, &\cos\theta\\
\cos\theta,& \rho\sin\theta
\end{pmatrix}
.
\end{split}
\end{equation}
Since $|\theta|<\delta\ll1$ and $|m|<4R$, there exists a small
constant $c>0$ which is independent of $r,m,\nu$ such that
\begin{equation*}
\begin{split}
\Big|\det\Big(\frac{\partial^2\tilde{\Phi}_{r,m,\nu}}{\partial(\rho,\theta)^2}\Big)\Big|=
\left|\frac {2m} r\rho\sin\theta-\cos^2\theta\right|\geq
\cos^2\theta-4m|\sin\theta|/r \geq c.
\end{split}
\end{equation*}
Then the modified phase function
$\tilde{\Phi}_{r,m,\nu}(\rho,\theta)$ is non-degenerate, the
standard stationary phase argument gives that there exists a
constant $C>0$ which is independent of $r,m,\nu$ such that
\begin{equation}\label{3.34}
\begin{split}
|\psi_{m}^\nu(r)|\leq C r^{-1}.
\end{split}
\end{equation}
For fixed $t, R$, we define $A=\{j\in\Z: |t+\frac j4|\leq 4R\}$. It
is easy to see  $\sharp A$ is $O(R)$. Thus it follows from
\eqref{3.34} and the Cauchy-Schwarz inequality that
\begin{equation}\label{3.35}
\begin{split}
&\Big\|r^{-\frac{n-2}2}\Big(\sum_{\nu\in\chi_K}\sum_{\ell=1}^{d(\nu)}\big|\sum_{j\in
A} b_{\nu,\ell}^j \psi_{t+\frac j4}^\nu(r)\big|^2\Big)^{\frac12}
\Big\|_{L^\infty_t(\R;L^\infty_{\mu(r)}([R,2R]))}\\&\leq
C_{\delta,N}R^{-\frac{n-1}2}\Big(\sum_{\nu\in\chi_K}\sum_{\ell=1}^{d(\nu)}\sum_{j}
|b_{\nu,\ell}^j|^2 \Big)^{\frac12}\lesssim
R^{-\frac{n-1}2}\Big\|\big(\sum_{\nu\in\chi_K}\sum_{\ell=1}^{d(\nu)}
\big|b_{\nu,\ell}(\rho)\big|^2\big)^{\frac
12}\beta(\rho)\Big\|_{L^{2}_{\mu(\rho)}}.
\end{split}
\end{equation}
 Together with \eqref{3.31}, this gives \eqref{3.27}. Thus it proves
\eqref{3.10}.\vspace{0.2cm}

{\bf Step 3.} We prove \eqref{3.11} and \eqref{3.13}, i.e. the case
$q=3p'$ and $2\leq p\leq4$. The \eqref{3.12} follows from the
interpolation of \eqref{3.11} and \eqref{3.9}. To do so, we need to
use the bilinear argument to explore the oscillation both in
$e^{it\rho^2}$ and the Bessel function $J_{\nu}(r\rho)$. For our
purpose, we have to use the complete asymptotic formula for the
Bessel function \cite{Stein1,Watson} and verify the sum of the
coefficient is absolutely convergent when $\nu^2\ll r$. On the other
hand the Hardy-Littlewood-Sobolev inequality fails at $q=4$, we
require the Whitney-type decomposition to overcome this difficulty.

To prove \eqref{3.11} and \eqref{3.13},  it suffices to prove: for
$q=3p'$ and $2\leq p\leq4$
\begin{equation}\label{3.36}
\begin{split}
\Big\|\Big(\sum_{\nu\in\chi_K}\sum_{\ell=1}^{d(\nu)}\big|\int_0^\infty
e^{it\rho^2} &J_{\nu}(
r\rho)b_{\nu,\ell}(\rho)\beta(\rho)\rho^{\frac
n2}\mathrm{d}\rho\big|^2\Big)^{\frac12} \Big\|_{L^q_{t,r}(\R\times
[R,2R])}\\&\lesssim R^{-\frac
12+\epsilon_q}\Big\|\Big(\sum_{\nu\in\chi_K}\sum_{\ell=1}^{d(\nu)}
(1+\nu)^{\frac{4}{q}}\big|b_{\nu,\ell}(\rho)\big|^2\Big)^{\frac
12}\beta(\rho)\Big\|_{L^{p}_{\mu(\rho)}(I)},
\end{split}
\end{equation}
where $\epsilon_q=\epsilon$ if $q=4$ otherwise $\epsilon_q=0$.

$\bullet$ Case 1: $\nu\in \Omega_1:=\{\nu\in\chi_K: R \ll \nu\}$.

By the Minkowski inequality,  \eqref{2.8} and the Hausdorff-Young
inequality in $t$, it shows that
\begin{equation}\label{4.24}
\begin{split}
\Big\|\Big(\sum_{\nu\in \Omega_1}\sum_{\ell=1}^{d(\nu)}
&\big|\int_0^\infty e^{it\rho^2}J_{\nu}(
r\rho)b_{\nu,\ell}(\rho)\beta(\rho)\rho^{\frac
n2}\mathrm{d}\rho\big|^2\Big)^{\frac12} \Big\|_{L^q_{t,r}(\R\times
[R,2R])}\\&\lesssim \Big(\sum_{\nu\in
\Omega_1}\sum_{\ell=1}^{d(\nu)} \big\|J_{\nu}(
r\rho)b_{\nu,\ell}(\rho)\rho^{\frac
n2-{\frac1q}}\beta(\rho)\big\|_{L^{q'}_\rho
L^{q}_r([R,2R])}^2\Big)^{\frac12}\\&\lesssim
\Big(\sum_{\nu\in\chi_K}\sum_{\ell=1}^{d(\nu)}
\big\|e^{-cr}b_{\nu,\ell}(\rho)\rho^{\frac
n2-{\frac1q}}\beta(\rho)\big\|_{L^{q'}_\rho
L^{q}_r([R,2R])}^2\Big)^{\frac12}\\&\lesssim
Ce^{-cR}\Big\|\big(\sum_{\nu\in\chi_K}\sum_{\ell=1}^{d(\nu)}
\big|b_{\nu,\ell}(\rho)\big|^2\big)^{\frac
12}\beta(\rho)\Big\|_{L^{p}_{\mu(\rho)}(I)}.
\end{split}
\end{equation}

$\bullet$ Case 2: $\nu\in \Omega_2:=\{\nu\in\chi_K: \nu\lesssim R
\lesssim \nu^2\}$.

By \eqref{3.8}, we have by canceling some $r$-weights
\begin{equation}\label{4.25}
\begin{split}
\Big\|\Big(\sum_{\nu\in\Omega_2}\sum_{\ell=1}^{d(\nu)}\big|\int_0^\infty
e^{it\rho^2} &J_{\nu}(
r\rho)b_{\nu,\ell}(\rho)\beta(\rho)\rho^{\frac
n2}\mathrm{d}\rho\big|^2\Big)^{\frac12} \Big\|_{L^2_{t,r}(\R\times
[R,2R])}\\&\lesssim R^{-\frac
12}\Big\|\Big(\sum_{\nu\in\chi_K}\sum_{\ell=1}^{d(\nu)}
(1+\nu)^{2}\big|b_{\nu,\ell}(\rho)\big|^2\Big)^{\frac
12}\Big\|_{L^{2}_{\mu(\rho)}(I)}.
\end{split}
\end{equation}
On the other hand, we obtain by \eqref{3.10}
\begin{equation*}
\begin{split}
\Big\|\Big(\sum_{\nu\in
\Omega_2}\sum_{\ell=1}^{d(\nu)}\big|\int_0^\infty e^{it\rho^2}
J_{\nu}( r\rho)b_{\nu,\ell}(\rho)&\beta(\rho)\rho^{\frac
n2}\mathrm{d}\rho\big|^2\Big)^{\frac12}
\Big\|_{L^\infty_{t,r}(\R\times [R,2R])}\\&\lesssim
R^{-\frac12}\Big\|\big(\sum_{\nu\in\chi_K}\sum_{\ell=1}^{d(\nu)}
\big|b_{\nu,\ell}(\rho)\big|^2\big)^{\frac
12}\beta(\rho)\Big\|_{L^{2}_{\mu(\rho)}(I)}.
\end{split}
\end{equation*}
Interpolating this with \eqref{4.25}, we have
\begin{equation}\label{4.28}
\begin{split}
\Big\|\Big(\sum_{\nu\in
\Omega_2}\sum_{\ell=1}^{d(\nu)}\big|\int_0^\infty e^{it\rho^2}
&J_{\nu}( r\rho)b_{\nu,\ell}(\rho)\beta(\rho)\rho^{\frac
n2}\mathrm{d}\rho\big|^2\Big)^{\frac12} \Big\|_{L^q_{t,r}(\R\times
[R,2R])}\\&\lesssim R^{-\frac
12}\Big\|\Big(\sum_{\nu\in\chi_K}\sum_{\ell=1}^{d(\nu)}
(1+\nu)^{\frac4q}\big|b_{\nu,\ell}(\rho)\big|^2\Big)^{\frac
12}\Big\|_{L^{2}_{\mu(\rho)}(I)}.
\end{split}
\end{equation}

$\bullet$ Case 3: $\nu\in \Omega_3:=\{\nu\in\chi_K:  \nu^2\ll R\}$.

To prove \eqref{3.36} in this case, since the $\nu$-weight is large
than $1$, it suffices to show
\begin{equation}\label{4.29}
\begin{split}
\Big\|\sum_{\nu\in \Omega_3}\sum_{\ell=1}^{d(\nu)}
\big|\int_{I}e^{it\rho^2}b_{\nu,\ell}(\rho)& J_{\nu}(\rho
r)\beta(\rho)\rho^{\frac n2}{\mathrm{d}\rho}\big|^2\Big\|_{L^{\frac
q2}_{t,r}(\R\times [R,2R])}\\&\lesssim
R^{-1+\epsilon_q}\Big\|\big(\sum_{\nu\in\chi_K}\sum_{\ell=1}^{d(\nu)}
\big|b_{\nu,\ell}(\rho)\big|^2\big)^{\frac
12}\beta(\rho)\Big\|^2_{L^{p}_{\mu(\rho)}(I)}.
\end{split}
\end{equation}
To this end, let $\widetilde{\beta}(\rho)=\beta(\rho)\rho^{\frac
n2}$, we rewrite
\begin{equation*}
\begin{split}
&\sum_{\nu\in \Omega_3}\sum_{\ell=1}^{d(\nu)} \big|\int_{I}e^{
it\rho^2}b_{\nu,\ell}(\rho)J_{\nu}(\rho r)\beta(\rho)\rho^{\frac
n2}{\mathrm{d}\rho}\big|^2\\&\qquad=\sum_{\nu\in
\Omega_3}\sum_{\ell=1}^{d(\nu)} \int_{I}e^{
it\rho_1^2}b_{\nu,\ell}(\rho_1)J_{\nu}(\rho_1
r)\widetilde{\beta}(\rho_1){\mathrm{d}\rho_1}\\&\qquad\qquad\qquad\qquad\qquad\qquad\int_{I}e^{-
it\rho_2^2}\overline{b_{\nu,\ell}(\rho_2)J_{\nu}(\rho_2
r)}\widetilde{\beta}(\rho_2){\mathrm{d}\rho_2}.
\end{split}
\end{equation*}

$\bullet$ Subcase (a): $q=3p'$ with $2\leq p<4$.  Before proving
\eqref{4.29}, we recall a complete asymptotic formula for the Bessel
function \cite{Stein1,Watson}. When $\nu$ is fixed, the complete
asymptotic formula for $J_{\nu}(\rho r)$, as $r\rightarrow\infty$,
is
\begin{equation}\label{4.30}
\begin{split}
J_{\nu}(\rho r)\sim&\big(\rho r\big)^{-\frac12}\cos\big(\rho
r-\frac{\nu\pi}2-\frac\pi4\big)\sum_{m=0}^\infty(\rho
r)^{-2m}a_m(\nu)\\&+\big(\rho r\big)^{-\frac12}\sin\big(\rho
r-\frac{\nu\pi}2-\frac\pi4\big)\sum_{m=0}^{\infty}(\rho
r)^{-2m-1}b_m(\nu)
\end{split}
\end{equation}
where \begin{equation*}
\begin{split}
a_m(\nu)=\frac{(-1)^m\Gamma(\nu+\frac12+2m)}{2^{2m}(2m)!\cdot\Gamma(\nu+\frac12-2m)},~
~b_m(\nu)=\frac{(-1)^m\Gamma(\nu+\frac32+2m)}{2^{(2m+1)}(2m+1)!\cdot\Gamma(\nu-\frac12-2m)}.
\end{split}
\end{equation*}
Now we aim to estimate
\begin{equation*}
\begin{split}
&\sum_{\nu\in \Omega_3}\sum_{\ell=1}^{d(\nu)} \big|\int_{I}e^{
it\rho^2}b_{\nu,\ell}(\rho)J_{\nu}(\rho r)\beta(\rho)\rho^{\frac
n2}{\mathrm{d}\rho}\big|^2\\&\sim r^{-1}\sum_{\nu\in
\Omega_3}e^{i\nu\pi}\sum_{\ell=1}^{d(\nu)}\sum_{m_1=0}^\infty\sum_{m_2=0}^\infty
r^{-2(m_1+m_2)}a_{m_1}(\nu)a_{m_2}(\nu) \\&\quad\int_{I\times I}e^{
it(\rho_1^2-\rho_2^2)}b_{\nu,\ell}(\rho_1)\overline{b_{\nu,\ell}(\rho_2)}
\widetilde{\beta}(\rho_1)\widetilde{\beta}(\rho_2) e^{-
ir(\rho_1\pm\rho_2)}\rho_1^{-2m_1-\frac12}\rho_2^{-2m_2-\frac12}{\mathrm{d}\rho_1}{\mathrm{d}\rho_2}\\&\qquad+\text{similar
terms}.
\end{split}
\end{equation*}
Since the similar terms can be estimated by the same argument, we
only estimate
\begin{equation*}
\begin{split}
&\Big\| r^{-1}\sum_{m_1,m_2=0}^\infty(2\pi
r)^{-2(m_1+m_2)}\int_{I\times I}e^{ it(\rho_1^2-\rho_2^2)}e^{-
ir(\rho_1\pm\rho_2)}\\&\quad\sum_{\nu\in
\Omega_3}\sum_{\ell=1}^{d(\nu)}e^{i\nu\pi}a_{m_1}(\nu)a_{m_2}(\nu)b_{\nu,\ell}(\rho_1)\overline{b_{\nu,\ell}(\rho_2)}
\widetilde{\beta}(\rho_1)\widetilde{\beta}(\rho_2)\rho_1^{-2m_1-\frac12}\rho_2^{-2m_2-\frac12}
{\mathrm{d}\rho_1}{\mathrm{d}\rho_2}\Big\|_{L^{\frac
q2}_{t,r}(\R\times [R,2R])}.
\end{split}
\end{equation*}
Let $$s_1=\rho_1\pm\rho_2,\quad s_2=\rho_1^2-\rho_2^2$$ and
$\Omega\subset \R\times\R$ be the image of $I \times I$ under such
change of variables. Then by changing variables, we need estimate
\begin{equation*}
\begin{split}
\Big\|r^{-1}\sum_{m_1,m_2=0}^\infty r^{-2(m_1+m_2)}
&\Big(\int_{\Omega}e^{i(ts_2+rs_1) }\sum_{\nu\in
\Omega_3}\sum_{\ell=1}^{d(\nu)}a_{m_1}(\nu)a_{m_2}(\nu)b_{\nu,\ell}(\rho_1)\overline{b_{\nu,\ell}(\rho_2)}
\\
&\times\frac{\widetilde{\beta}(\rho_1)\widetilde{\beta}(\rho_2)\rho_1^{-2m_1-\frac12}\rho_2^{-2m_2-\frac12}}{|\rho_1\pm\rho_2|}
{\mathrm{d}s_1}{\mathrm{d}s_2}\Big)\Big\|_{L^{\frac
q2}_{t,r}(\R\times [R,2R])}.
\end{split}
\end{equation*}
Since $q> 4$, by the Hausdorff-Young inequality, it suffices to show
\begin{equation*}
\begin{split}
&\sum_{m_1,m_2=0}^\infty(2\pi R)^{-2(m_1+m_2)}\Big\|\sum_{\nu\in
\Omega_3}\sum_{\ell=1}^{d(\nu)}a_{m_1}(\nu)a_{m_2}(\nu)
b_{\nu,\ell}(\rho_1)\overline{b_{\nu,\ell}(\rho_2)}\\&\qquad\qquad\qquad\qquad\qquad
\times\frac{\widetilde{\beta}(\rho_1)\widetilde{\beta}(\rho_2)\rho_1^{-2m_1-\frac12}\rho_2^{-2m_2-\frac12}}{|\rho_1\pm\rho_2|}
\Big\|_{L^{\frac
q{q-2}}_{s_1,s_2}(\Omega)}\\&\qquad\qquad\qquad\qquad\qquad\lesssim
\Big\|\big(\sum_{\nu\in\chi_K}\sum_{\ell=1}^{d(\nu)}
\big|b_{\nu,\ell}(\rho)\big|^2\big)^{\frac
12}\rho^{\frac{n-1}p}\Big\|^2_{L^{p}_{\mu(\rho)}(I)}.
\end{split}
\end{equation*}
By changing variables back, it reduces to prove
\begin{equation*}
\begin{split}
\sum_{m_1,m_2=0}^\infty(2\pi R)^{-2(m_1+m_2)}&\Big\|\sum_{\nu\in
\Omega_3}\sum_{\ell=1}^{d(\nu)}a_{m_1}(\nu)a_{m_2}(\nu)
b_{\nu,\ell}(\rho_1)\overline{b_{\nu,\ell}(\rho_2)}\frac{\widetilde{\beta}(\rho_1)\widetilde{\beta}(\rho_2)}{|\rho_1\pm\rho_2|^{\frac2q}}
\Big\|_{L^{\frac{q}{q-2}}_{\rho_1,\rho_2}(I^2)}\\&\qquad\qquad\qquad\lesssim
\Big\|\big(\sum_{\nu\in\chi_K}\sum_{\ell=1}^{d(\nu)}
\big|b_{\nu,\ell}(\rho)\big|^2\big)^{\frac
12}\rho^{\frac{n-1}p}\Big\|^2_{L^{p}_{\rho}(I)}.
\end{split}
\end{equation*}
Recalling
$$a_m(\nu)=\frac{(-1)^m\Gamma(\nu+\frac12+2m)}{2^{2m}(2m)!\cdot\Gamma(\nu+\frac12-2m)},$$
it gives that
$$\sup_{\nu\in \Omega_3}|a_m(\nu)|=\frac{\Gamma(\sqrt{R}+\frac12+2m)}{2^{2m}(2m)!\cdot\Gamma(\sqrt{R}+\frac12-2m)}.$$
On the other hand, we have the uniformly estimate
\begin{equation*}
\begin{split}
\sum_{m=0}^\infty(2\pi
R)^{-2m}\frac{\Gamma(\sqrt{R}+\frac12+2m)}{2^{2m}(2m)!\cdot\Gamma(\sqrt{R}+\frac12-2m)}\leq
C.
\end{split}
\end{equation*}
Thus it suffices to prove
\begin{equation*}
\begin{split}
\Big\|\sum_{\nu\in\chi_K}\sum_{\ell=1}^{d(\nu)}
|b_{\nu,\ell}(\rho_1)\overline{b_{\nu,\ell}(\rho_2)}|\frac{\widetilde{\beta}(\rho_1)\widetilde{\beta}(\rho_2)}{|\rho_1\pm\rho_2|^{2/q}}
\Big\|_{L^{\frac q{q-2}}_{\rho_1,\rho_2}(I^2)}\lesssim
\Big\|\big(\sum_{\nu\in\chi_K}\sum_{\ell=1}^{d(\nu)}
\big|b_{\nu,\ell}(\rho)\big|^2\big)^{\frac
12}\rho^{\frac{n-1}p}\Big\|^2_{L^{p}_{\mu(\rho)}(I)}.
\end{split}
\end{equation*}
Since $p>\frac{q}{q-2}$ and $|\rho_1+\rho_2|\geq1$, the case
concerning $|\rho_1+\rho_2|$ is obvious to be proved. By the
Cauchy-Schwarz inequality, it is enough  to prove
\begin{equation}\label{4.31}
\begin{split}
\Big\|\int_{I}
\Big(\sum_{\nu\in\chi_K}\sum_{\ell=1}^{d(\nu)}|b_{\nu,\ell}(\rho_2)|^2\big)^{1/2}\frac{1}{|\rho_1-\rho_2|^{2/q}}
&\Big)^{\frac
q{q-2}}\mathrm{d}\rho_2\Big\|^{\frac{q-2}q}_{L^{[\frac{q-2}q-\frac1p]^{-1}\frac{q-2}q}_{\rho_1}(I)}\\&\lesssim
\Big\|\big(\sum_{\nu\in\chi_K}\sum_{\ell=1}^{d(\nu)}
\big|b_{\nu,\ell}(\rho)\big|^2\big)^{\frac
12}\Big\|_{L^{p}_{\rho}(I)}.
\end{split}
\end{equation}
Since assuming $q=3p'>4$, we have
$$1+\frac{q}{q-2}(\frac{q-2}q-\frac1p)=\frac{q}{q-2}\frac2q+\frac1p\frac q{q-2}.$$
Then \eqref{4.31} follows from the Hardy-Littlewood-Sobolve
inequality.\vspace{0.2cm}

$\bullet$ Subcase (b): $q=4$ and $p=4$. In this subcase, the
Hardy-Littlewood-Sobolev inequality fails, we cannot use the above
argument to prove \eqref{4.29}. We need a Whitney-type decomposition
to $I$. Performing a Whitney decomposition to $I$, for each $j\geq
0$, we break up $I$ into $O(2^j)$ dyadic intervals $Q_{\bar k}^j$ of
length $2^{-j}$ and also define $Q_{\bar k}^j \simeq Q_{\bar k'}^j$
if they are cousins, i.e. $Q_{\bar k}^j$ and $Q_{\bar k'}^j$ are not
adjacent but have adjacent parents. Then by \eqref{2.17}, we can
write the above as the following decomposition
\begin{equation*}
\begin{split}
\sum_{\nu\in\Omega_3}\sum_{\ell=1}^{d(\nu)}
\sum_{j\geq0}\sum_{\bar{k}}\sum_{\bar{k}':Q_{\bar k}^j \simeq
Q_{\bar k'}^j}F_{\bar{k}}^j~\overline{G_{\bar{k}'}^j},
\end{split}
\end{equation*}
where
\begin{equation*}
\begin{split}
F_{\bar{k}}^j=F_{\bar{k}}^j(t,r)= \int_{Q_{\bar k}^j}e^{
it\rho_1^2}b_{\nu,\ell}(\rho_1)J_{\nu}(\rho_1
r)\widetilde{\beta}(\rho_1){\mathrm{d}\rho_1},
\end{split}
\end{equation*}
and
\begin{equation*}
\begin{split}
G_{\bar{k}'}^j=G_{\bar{k}'}^j(t,r)= \int_{Q_{{\bar k}'}^j}e^{
it\rho_2^2}b_{\nu,\ell}(\rho_2)J_{\nu}(\rho_2
r)\widetilde{\beta}(\rho_2){\mathrm{d}\rho_2}.
\end{split}
\end{equation*}
Thus by triangle inequality and $\rho\in[1,2]$, it suffices to prove
\begin{equation}\label{4.32}
\begin{split}
\sum_{j\geq \log R}\Big\|\sum_{k\in \Omega_3}\sum_{\ell=1}^{d(\nu)}
\sum_{\bar{k}}\sum_{\bar{k}':Q_{\bar k}^j \simeq Q_{\bar
k'}^j}&F_{\bar{k}}^j~\overline{G_{\bar{k}'}^j}\Big\|_{L^{2}_{t,r}(\R\times
[R,2R])}\\&\lesssim
R^{-1}\Big\|\big(\sum_{\nu\in\chi_K}\sum_{\ell=1}^{d(\nu)}
\big|b_{\nu,\ell}(\rho)\big|^2\big)^{\frac
12}\Big\|^2_{L^{4}_{\rho}(I)},
\end{split}
\end{equation}
and
\begin{equation}\label{4.33}
\begin{split}
\sum_{j\leq\log R}\Big\|\sum_{k\in \Omega_3}\sum_{\ell=1}^{d(\nu)}
\sum_{\bar{k}}\sum_{\bar{k}':Q_{\bar k}^j \simeq Q_{\bar
k'}^j}&F_{\bar{k}}^j~\overline{G_{\bar{k}'}^j}\Big\|_{L^{2}_{t,r}(\R\times
[R,2R])}\\&\lesssim
R^{-1+\epsilon}\Big\|\big(\sum_{\nu\in\chi_K}\sum_{\ell=1}^{d(\nu)}
\big|b_{\nu,\ell}(\rho)\big|^2\big)^{\frac
12}\Big\|^2_{L^{4}_{\rho}(I)}.
\end{split}
\end{equation}
Firstly, we prove \eqref{4.32}. To this end, by the Cauchy-Schwarz
inequality and the triangle inequality, it follows
\begin{equation}\label{4.34}
\begin{split}
\text{LHS of }~\eqref{4.32}\lesssim \sum_{j\geq \log
R}\sum_{\bar{k}}\sum_{\bar{k}':Q_{\bar k}^j \simeq Q_{\bar
k'}^j}&\Big\|\big(\sum_{\nu\in \Omega_3}\sum_{\ell=1}^{d(\nu)}
|F_{\bar{k}}^j|^2\big)^{\frac12}\Big\|_{L^{2}_{t,r}(\R\times
[R,2R])}\\&\times\Big\|\big(\sum_{\nu\in
\Omega_3}\sum_{\ell=1}^{d(\nu)}
|G_{\bar{k}'}^j|^2\big)^{\frac12}\Big\|_{L^{\infty}_{t,r}(\R\times
[R,2R])}.
\end{split}
\end{equation}
By  \eqref{2.14}, the Minkowski inequality, H\"older's inequality
and the Hausdorff-Young inequality in $t$, we have by arguing as
before
\begin{equation}\label{4.35}
\begin{split}
&\Big\|\big(\sum_{\nu\in \Omega_3}\sum_{\ell=1}^{d(\nu)}
|G_{\bar{k}'}^j|^2\big)^{\frac12}\Big\|_{L^{\infty}_{t,r}(\R\times
[R,2R])}\\&\quad\lesssim\Big(\sum_{\nu\in
\Omega_3}\sum_{\ell=1}^{d(\nu)} \big\|\int_{Q_{{\bar k}'}^j}e^{
it\rho_2^2}b_{\nu,\ell}(\rho_2)J_{\nu}(\rho_2
r)\widetilde{\beta}(\rho_2){\mathrm{d}\rho_2}\big\|_{L^{\infty}_{t,r}(\R\times
[R,2R])}^2\Big)^{\frac12}\\&\quad\lesssim
R^{-\frac12}\Big(\sum_{\nu\in\chi_K}\sum_{\ell=1}^{d(\nu)}
\big\|b_{\nu,\ell}(\rho_2)\widetilde{\beta}(\rho_2)
\big\|_{L^{1}_{\rho_2}(Q_{{\bar
k}'}^j)}^2\Big)^{\frac12}\\&\quad\lesssim R^{-\frac12}|Q_{{\bar
k}'}^j|^{\frac1{2}}\Big\|\Big(\sum_{\nu\in\chi_K}\sum_{\ell=1}^{d(\nu)}
|b_{\nu,\ell}(\rho_2)|
^2\Big)^{\frac12}\Big\|_{L^{2}_{\rho_2}(Q_{{\bar k}'}^j)},
\end{split}
\end{equation}
where we make use of $\rho_2\in Q_{{\bar k}'}^j\subset[1,2]$. On the
other hand, the Hausdorff-Young inequality in $t$ and similar
argument as before imply that
\begin{equation}\label{4.36}
\begin{split}
&\Big\|\big(\sum_{\nu\in \Omega_3}\sum_{\ell=1}^{d(\nu)}
|F_{\bar{k}}^j|^2\big)^{\frac12}\Big\|_{L^{2}_{t,r}(\R\times
[R,2R])}\\&\quad= \Big(\sum_{\nu\in \Omega_3}\sum_{\ell=1}^{d(\nu)}
\Big\|\int_{Q_{\bar k}^j}e^{
it\rho_1^2}b_{\nu,\ell}(\rho_1)J_{\nu}(\rho_1r)\widetilde{\beta}(\rho_1){\mathrm{d}\rho_1}\Big\|_{L^{2}_{t,r}(\R\times
[R,2R])}^2\Big)^{\frac12}\\&\quad\lesssim
\Big\|\Big(\sum_{\nu\in\chi_K}\sum_{\ell=1}^{d(\nu)}
|b_{\nu,\ell}(\rho_1)|
^2\Big)^{\frac12}\Big\|_{L^{2}_{\rho_1}(Q_{{\bar k}}^j)}.
\end{split}
\end{equation}
Together with \eqref{4.34} and \eqref{4.35}, it gives
\begin{equation*}
\begin{split}
\text{RHS of }~\eqref{4.34}\lesssim R^{-\frac12}\sum_{j\geq \log
R}2^{-\frac j2}\sum_{\bar{k}}\sum_{\bar{k}':Q_{\bar k}^j \simeq
Q_{\bar k'}^j}&\Big\|\Big(\sum_{\nu\in\chi_K}\sum_{\ell=1}^{d(\nu)}
|b_{\nu,\ell}(\rho_1)|
^2\Big)^{\frac12}\Big\|_{L^{2}_{\rho_1}(Q_{{\bar
k}}^j)}\\&\Big\|\Big(\sum_{\nu\in\chi_K}\sum_{\ell=1}^{d(\nu)}
|b_{\nu,\ell}(\rho_2)|
^2\Big)^{\frac12}\Big\|_{L^{2}_{\rho_2}(Q_{{\bar k}'}^j)}.
\end{split}
\end{equation*}
Recalling the property of the Whitney decomposition that for each
fixed $\bar{k}$, there are only $O(1)$ cousins of $Q_{{\bar k}}^j$,
then we have
\begin{equation*}
\begin{split}
\text{RHS of }~\eqref{4.34}\lesssim R^{-1}
\Big\|\Big(\sum_{\nu\in\chi_K}\sum_{\ell=1}^{d(\nu)}
|b_{\nu,\ell}(\rho_2)| ^2\Big)^{\frac12}\Big\|^2_{L^{2}(I)}.
\end{split}
\end{equation*}
Thus we prove \eqref{4.32}.

Now we prove \eqref{4.33} to complete the proof. Recalling
\eqref{4.30} and the definitions of $F_{\bar{k}}^j$ and
$G_{\bar{k}'}^j$, now we aim to estimate
\begin{equation*}
\begin{split}
&\sum_{\nu\in \Omega_3}\sum_{\ell=1}^{d(\nu)}
\sum_{\bar{k}}\sum_{\bar{k}':Q_{\bar k}^j \simeq Q_{\bar
k'}^j}F_{\bar{k}}^j~\overline{G_{\bar{k}'}^j}\\&\sim
r^{-1}\sum_{\bar{k}}\sum_{\bar{k}':Q_{\bar k}^j \simeq Q_{\bar
k'}^j}\sum_{\nu\in
\Omega_3}e^{i\nu\pi}\sum_{\ell=1}^{d(\nu)}\sum_{m_1=0}^\infty\sum_{m_2=0}^\infty
r^{-2(m_1+m_2)}a_{m_1}(k)a_{m_2}(k) \\&\quad\int_{Q_{\bar k}^j
\times Q_{\bar k'}^j}e^{
it(\rho_1^2-\rho_2^2)}b_{\nu,\ell}(\rho_1)\overline{b_{\nu,\ell}(\rho_2)}
\widetilde{\beta}(\rho_1)\widetilde{\beta}(\rho_2) e^{-
ir(\rho_1\pm\rho_2)}\rho_1^{-2m_1-\frac12}\rho_2^{-2m_2-\frac12}{\mathrm{d}\rho_1}{\mathrm{d}\rho_2}\\&\qquad+\text{similar
terms}.
\end{split}
\end{equation*}
As before, since the similar terms can be estimated by the same
argument, we only consider
\begin{equation*}
\begin{split}
&\Big\| r^{-1}\sum_{m_1,m_2=0}^\infty
r^{-2(m_1+m_2)}\sum_{\bar{k}}\sum_{\bar{k}':Q_{\bar k}^j \simeq
Q_{\bar k'}^j} \int_{Q_{\bar k}^j \times Q_{\bar k'}^j}e^{
it(\rho_1^2-\rho_2^2)}e^{- ir(\rho_1\pm\rho_2)}\\&\quad\sum_{\nu\in
\Omega_3}e^{i\nu\pi}\sum_{\ell=1}^{d(\nu)}a_{m_1}(k)a_{m_2}(k)b_{\nu,\ell}(\rho_1)\overline{b_{\nu,\ell}(\rho_2)}
\widetilde{\beta}(\rho_1)\widetilde{\beta}(\rho_2)\rho_1^{-2m_1-\frac12}\rho_2^{-2m_2-\frac12}
{\mathrm{d}\rho_1}{\mathrm{d}\rho_2}\Big\|_{L^{2}_{t,r}(\R\times
[R,2R])}.
\end{split}
\end{equation*}
For this purpose, let $s_1=\rho_1\pm\rho_2,\quad
s_2=\rho_1^2-\rho_2^2$ and $\Omega_{\bar{k},\bar{k}'}^j\subset
\R\times\R$ be the image of $Q_{\bar k}^j \times Q_{\bar k'}^j$
under such change of variables. Then we aim to estimate
\begin{equation*}
\begin{split}
&\sum_{j\leq\log R}\Big\|r^{-1}\sum_{m_1,m_2=0}^\infty
r^{-2(m_1+m_2)}\sum_{\bar{k}}\sum_{\bar{k}':Q_{\bar k}^j \simeq
Q_{\bar k'}^j}\Big(\int_{\Omega_{\bar{k},\bar{k}'}^j}e^{i(ts_2+rs_1)
}\\
&\times\sum_{\nu\in
\Omega_3}e^{i\nu\pi}\sum_{\ell=1}^{d(k)}a_{m_1}(k)a_{m_2}(k)b_{\nu,\ell}(\rho_1)\overline{b_{\nu,\ell}(\rho_2)}
\frac{\widetilde{\beta}(\rho_1)\widetilde{\beta}(\rho_2)\rho_1^{-2m_1-\frac12}\rho_2^{-2m_2-\frac12}}{|\rho_1\pm\rho_2|}
{\mathrm{d}s_1}{\mathrm{d}s_2}\Big)\Big\|_{L^{2}_{t,r}(\R\times
[R,2R])}.
\end{split}
\end{equation*}
To prove \eqref{4.29}, by the Hausdorff-Young inequality and the
quasi-orthogonality(see \cite[Lemma 6.1]{TVV}),  it suffices to
establish
\begin{equation*}
\begin{split}
\sum_{j\leq\log R} \sum_{m_1,m_2=0}^\infty
R^{-2(m_1+m_2)}&\Big(\sum_{\bar{k}}\sum_{\bar{k}':Q_{\bar k}^j
\simeq Q_{\bar k'}^j}\Big\|\sum_{\nu\in
\Omega_3}\sum_{\ell=1}^{d(\nu)}a_{m_1}(k)a_{m_2}(k)
b_{\nu,\ell}(\rho_1)\overline{b_{\nu,\ell}(\rho_2)}\\&\qquad\times\frac{\widetilde{\beta}(\rho_1)\widetilde{\beta}(\rho_2)\rho_1^{-2m_1-\frac12}\rho_2^{-2m_2-\frac12}}{|\rho_1-\rho_2|}
\Big\|^{2}_{L^2_{s_1,s_2}(\Omega_{\bar{k},\bar{k}'}^j)}\Big)^{\frac12}\\&\qquad\lesssim
R^{\epsilon}\Big\|\big(\sum_{\nu\in\chi_K}\sum_{\ell=1}^{d(\nu)}
\big|b_{\nu,\ell}(\rho)\big|^2\big)^{\frac
12}\Big\|^2_{L^{4}_{\rho}(I)}.
\end{split}
\end{equation*}
By changing variables back, it reduces to prove
\begin{equation*}
\begin{split}
&\sum_{j\leq\log R} \sum_{m_1,m_2=0}^\infty
R^{-2(m_1+m_2)}\\&\Big(\sum_{\bar{k}}\sum_{\bar{k}':Q_{\bar k}^j
\simeq Q_{\bar k'}^j}\Big\|\sum_{\nu\in
\Omega_3}\sum_{\ell=1}^{d(\nu)}a_{m_1}(k)a_{m_2}(k)
b_{\nu,\ell}(\rho_1)\overline{b_{\nu,\ell}(\rho_2)}\frac{\widetilde{\beta}(\rho_1)\widetilde{\beta}(\rho_2)}{|\rho_1\pm\rho_2|^{\frac12}}
\Big\|^2_{L^{2}_{\rho_1,\rho_2}(Q_{\bar k}^j \times Q_{\bar
k'}^j)}\Big)^{\frac12}\\&\qquad\qquad\qquad\qquad\qquad\lesssim
R^{\epsilon}\Big\|\big(\sum_{\nu\in\chi_K}\sum_{\ell=1}^{d(\nu)}
\big|b_{\nu,\ell}(\rho)\big|^2\big)^{\frac
12}\Big\|^2_{L^{4}_{\rho}(I)}.
\end{split}
\end{equation*}
As before, we also have the uniformly estimate
\begin{equation*}
\begin{split}
\sum_{m=0}^\infty
R^{-2m}\frac{\Gamma(\sqrt{R}+\frac12+2m)}{2^{2m}(2m)!\cdot\Gamma(\sqrt{R}+\frac12-2m)}\leq
C.
\end{split}
\end{equation*}
Thus it suffices to prove
\begin{equation*}
\begin{split}
&\sum_{j\leq\log R}\Big(\sum_{\bar{k}}\sum_{\bar{k}':Q_{\bar k}^j
\simeq Q_{\bar k'}^j}\Big\|\sum_{\nu\in
\Omega_3}\sum_{\ell=1}^{d(\nu)}
|b_{\nu,\ell}(\rho_1)\overline{b_{\nu,\ell}(\rho_2)}|\frac{\beta(\rho_1)\beta(\rho_2)}{|\rho_1\pm\rho_2|^{\frac12}}
\Big\|^2_{L^{2}_{\rho_1,\rho_2}(Q_{\bar k}^j \times Q_{\bar
k'}^j)}\Big)^{\frac12}\\&\qquad\qquad\qquad\qquad\qquad\lesssim
R^{\epsilon}\Big\|\big(\sum_{\nu\in\chi_K}\sum_{\ell=1}^{d(\nu)}
\big|b_{\nu,\ell}(\rho)\big|^2\big)^{\frac
12}\Big\|^2_{L^{4}_{\rho}(I)}.
\end{split}
\end{equation*}
By the Cauchy-Schwarz inequality and $\text{dist}(Q_{\bar k}^j,
Q_{\bar k'}^j)\geq 2^{-j}$, we need to prove
\begin{equation*}
\begin{split}
\sum_{j\leq\log R}2^{\frac {j}2}
\Big(\sum_{\bar{k}}&\sum_{\bar{k}':Q_{\bar k}^j \simeq Q_{\bar
k'}^j}\Big\|\Big(\sum_{\nu\in \Omega_3}\sum_{\ell=1}^{d(\nu)}
|b_{\nu,\ell}(\rho_1)|^2\Big)^{\frac12}
\Big\|^{2}_{L^{2}_{\rho_1}(Q_{\bar k}^j )}\Big\|\Big(\sum_{\nu\in
\Omega_3}\sum_{\ell=1}^{d(\nu)}
|b_{\nu,\ell}(\rho_2)|^2\Big)^{\frac12} \Big\|^{2}_{L^{2}_{\rho_2}(
Q_{\bar k'}^j)}\Big)^{\frac
12}\\&\qquad\qquad\qquad\qquad\qquad\lesssim
R^\epsilon\Big\|\big(\sum_{\nu\in\chi_K}\sum_{\ell=1}^{d(\nu)}
\big|b_{\nu,\ell}(\rho)\big|^2\big)^{\frac
12}\Big\|^2_{L^{4}_{\rho}(I)}.
\end{split}
\end{equation*}
Since $|Q_{\bar k}^j|=|Q_{{\bar k}'}^j|=2^{-j}$, by H\"older's
inequality, we can bound the left hand side by
\begin{equation*}
\begin{split}
\sum_{j\leq\log R}
&\Big(\sum_{\bar{k}}\Big\|\Big(\sum_{\nu\in\chi_K}\sum_{\ell=1}^{d(\nu)}
|b_{\nu,\ell}(\rho_1)|^2\Big)^{\frac12}
\Big\|^{4}_{L^{4}_{\rho_1}(Q_{\bar k}^j )}\Big)^{\frac 12}
\end{split}
\end{equation*}
Moreover it is controlled by
\begin{equation*}
\begin{split}
\log R\Big\|\big(\sum_{\nu\in\chi_K}\sum_{\ell=1}^{d(\nu)}
\big|b_{\nu,\ell}(\rho)\big|^2\big)^{\frac
12}\Big\|^2_{L^{4}_{\rho}(I)}\lesssim
R^\epsilon\Big\|\big(\sum_{\nu\in\chi_K}\sum_{\ell=1}^{d(\nu)}
\big|b_{\nu,\ell}(\rho)\big|^2\big)^{\frac
12}\Big\|^2_{L^{4}_{\rho}(I)}.
\end{split}
\end{equation*}
Hence it follows \eqref{4.33}. Therefore it completes the proof of
Proposition \ref{linear estimates}.

\section{Proof of the Theorem \ref{thm1}}

In this section, we utilize Proposition \ref{linear estimates} to
prove Theorem \ref{thm1}. We only need prove \eqref{1.6}. Indeed,
when the initial data $u_0=f(r)$ is radial Schwartz, so is the
Schwartz solution $u(t)$ by \eqref{2.22}. We can follow the argument
in proving \eqref{1.6} to easily obtain \eqref{1.5}, since the
$L^q$-norms on the compact set $\Sigma$ of a constant function are
equivalent for $1\leq q\leq\infty$. We remark that one need use
\eqref{3.5} to obtain \eqref{1.5} for $1\leq p\leq 2$.

Now we prove \eqref{1.6}. By the Sobolev embedding
$H^\alpha(\Sigma)\hookrightarrow L^q(\Sigma)$ with
$\alpha=(n-1)(\frac12-\frac1q)$, it suffices to show
\begin{equation}\label{4.1}
\begin{split}
&\|u(t,z)\|_{L^{{q}}_{t}L^{{q}}_{\mu(r)}L^2_{\theta}(\R\times
\R_+\times \Sigma)} \lesssim \|\mathcal{F}_{H}
{\big((1-\widetilde{\Delta}_h)^{\frac{1}{qn}}u_0\big)}\|_{L^p(M)}
\end{split}
\end{equation}
holds for the conditions $q>\frac{2(n+1)}{n}$ and
$\frac{n+2}q=\frac{n}{p'}$ with $p\geq 2$. By \eqref{2.22}, we have
the dyadic decomposition
\begin{equation}\label{4.2}
\begin{split}
&\|u(t,z)\|_{L^{{q}}_{t}L^{{q}}_{\mu(r)}L^2_{\theta}(\R\times
\R_+\times \Sigma)}\\&\lesssim
\big\|\sum_{\nu\in\chi_\infty}\sum_{\ell=1}^{d(\nu)}\varphi_{\nu,\ell}(\theta)\mathcal{H}_{\nu}\big[e^{it\rho^2}
b_{\nu,\ell}(\rho)\big](r)\big\|_{L^q_t(\R;L^q_{\mu(r)}(\R_+;L^{2}_\theta(\Sigma)))}
\\&\lesssim
\Big(\sum_{R}\big(\sum_{N}\big\|\sum_{\nu\in\chi_\infty}\sum_{\ell=1}^{d(\nu)}\varphi_{\nu,\ell}
(\theta)\\&\quad\times\int_0^\infty(r\rho)^{-\frac{n-2}2}J_{\nu}(r\rho)e^{
it\rho^2}b_{\nu,\ell}(\rho)\rho^{n-1}\beta(\frac\rho
N)\mathrm{d}\rho\big\|_{L^q_t(\R;L^q_{\mu(r)}([R,2R];L^{2}_\theta(\Sigma)))}\big)^q\Big)^{\frac1q}
\end{split}
\end{equation}
where $\beta \in C_c^\infty(\R)$ supported in $[1,2]$ and $R, N>0$
are dyadic numbers. Define
\begin{equation}\label{4.3}
\begin{split}
G(R,N;q):=\Big\|&\Big(\sum_{\nu\in\chi_\infty}\sum_{\ell=1}^{d(\nu)}\big|\int_0^\infty(r\rho)^{-\frac{n-2}2}\\&\quad\times
J_{\nu}(r\rho)e^{
it\rho^2}b_{\nu,\ell}(\rho)\rho^{n-1}\beta(\frac\rho
N)\mathrm{d}\rho\big|\Big)^{\frac12}\Big\|_{L^q_t(\R;L^q_{\mu(r)}([R,2R]))}.
\end{split}
\end{equation}
Now we use Proposition \ref{linear estimates}. As mentioned in
remarks after Proposition \ref{linear estimates}, we can replace
$\chi_K$ by $\chi_\infty$. By scaling argument and \eqref{3.1}, we
have
\begin{equation}\label{4.4}
\begin{split}
&G(R,N;2)\\&\lesssim \min\{(RN)^\frac12, (RN)^{\frac
n2}\}N^{n-\frac{n+2}2-\frac
n2}\Big\|\Big(\sum_{\nu\in\chi_\infty}\sum_{\ell=1}^{d(\nu)}|b_{\nu,\ell}(\rho)|^2\Big)^{\frac12}\beta(\frac\rho
N)\Big\|_{L^{2}_{\mu(\rho)}}.
\end{split}
\end{equation}
On the other hand, for $\bar{q}=3\bar{p}'$ and $2\leq \bar{p}<4$, we
have by \eqref{3.4}
\begin{equation}\label{4.5}
\begin{split}
G(R,N;\bar{q})\lesssim&\min\big\{(RN)^{(n-1)(\frac1{\bar{q}}-\frac12)},
(RN)^{\frac n{\bar{q}}}\big\}N^{n-\frac{n+2}{\bar{q}}-\frac
n{\bar{p}}}\\&\times\Big\|\Big(\sum_{\nu\in\chi_\infty}\sum_{\ell=1}^{d(\nu)}(1+\nu)^{\frac4{\bar
q}}|b_{\nu,\ell}(\rho)|^2\Big)^{\frac12}\beta(\frac\rho
N)\Big\|_{L^{\bar p}_{\mu(\rho)}}.
\end{split}
\end{equation}
Applying interpolation theorem to \eqref{4.4} and \eqref{4.5} with
index $\delta=2-\frac3q-\frac1p$,
$$\frac1q=\frac{1-\delta}2+\frac{\delta}{\bar{q}},\qquad \frac1p=\frac{1-\delta}2+\frac{\delta}{\bar{p}}$$
where ${\bar{q}}=3{\bar{p}}'$, we hence have for
$\frac{n+2}q=\frac{n}{p'}$,
\begin{equation}\label{4.6}
\begin{split}
G(R,N;{q})\lesssim&\min\big\{(RN)^{\frac n q},
(RN)^{-\frac{n-1}2[1-\frac{2(n+1)}{qn}]}\big\}\\ \quad\quad&\times
\Big\|\Big(\sum_{\nu\in\chi_\infty}\sum_{\ell=1}^{d(\nu)}(1+\nu)^{\frac4{qn}}
|b_{\nu,\ell}(\rho)|^2\Big)^{\frac12}\beta(\frac\rho
N)\Big\|_{L^{p}_{\mu(\rho)}}.
\end{split}
\end{equation}
Combining \eqref{4.2} with \eqref{4.6}, we have
\begin{equation*}
\begin{split}
\|u(t,z)\|_{L^{{q}}_{t}L^{{q}}_{\mu(r)}L^2_{\theta}(\R\times
\R_+\times \Sigma)} &\lesssim
\Big(\sum_{R}\Big(\sum_{N}\min\big\{(RN)^{\frac n q},
(RN)^{-\frac{n-1}2[1-\frac{2(n+1)}{qn}]}\big\}\\&\times\Big\|\Big(\sum_{\nu\in\chi_\infty}\sum_{\ell=1}^{d(\nu)}(1+\nu)^{\frac4{qn}}
|b_{\nu,\ell}(\rho)|^2\Big)^{\frac12}\beta(\frac\rho
N)\Big\|_{L^{p}_{\mu(\rho)}}\Big)^q\Big)^{\frac1q}.
\end{split}
\end{equation*}
Since $q>\frac{2(n+1)}n$ and $R, N$ are both dyadic number, we have
\begin{equation*}
\begin{split}
&\sup_{R>0}\sum_N \min\left\{(RN)^{\frac n q},
(RN)^{-\frac{n-1}2[1-\frac{2(n+1)}{qn}]}\right\}<\infty,\\
&\sup_{N>0}\sum_R \min\left\{(RN)^{\frac n q},
(RN)^{-\frac{n-1}2[1-\frac{2(n+1)}{qn}]}\right\}<\infty.
\end{split}
\end{equation*}
By using the Schur's test, for $p$ and $q$ where
$q>\frac{2(n+1)}n>p\geq2$, we have
\begin{equation*}
\begin{split}
&\|u(t,z)\|_{L^{{q}}_{t}L^{{q}}_{\mu(r)}L^2_{\theta}(\R\times
\R_+\times \Sigma)}\\&\lesssim
\Big(\sum_{N}\Big\|\Big(\sum_{\nu\in\chi_\infty}\sum_{\ell=1}^{d(\nu)}(1+\nu)^{\frac4{qn}}
|b_{\nu,\ell}(\rho)|^2\Big)^{\frac12}\beta(\frac\rho
N)\Big\|_{L^{p}_{\mu(\rho)}}^p\Big)^{\frac1q}\\& \lesssim
\|\mathcal{F}_{H}
{\big((1-\widetilde{\Delta}_h)^{\frac{1}{qn}}u_0\big)}\|_{L^p(M)}.
\end{split}
\end{equation*}
Therefore we prove \eqref{4.1}.

\begin{center}

\end{center}
\end{document}